\documentclass[12pt,reqno]{amsart}
\pdfoutput=1
\usepackage{a4wide}
\usepackage{subfigure}

\usepackage{amssymb,amsmath,amsthm,newlfont,enumerate}

\def\paragraph#1{{\bf #1\ }}

 \def\OO{\rm \hbox{O\kern-.34em\raise.47ex
         \hbox{$\scriptscriptstyle |$}\kern-.46em\raise.47ex
         \hbox{$\scriptscriptstyle |$}\kern+0.5 em }}

\def\hcboxcm#1#2{\hbox to #1{\hfill #2 \hfill}}

\def\null{\hbox{}}
\let\eps\varepsilon

\def\tn1{\widetilde n_1}
\def\tn2{\widetilde n_2}
\def\tn{\widetilde n }

\let\ds\displaystyle

\def\be{\begin{equation}}
\def\ee{\end{equation}}
\def\bea{\begin{eqnarray}}
\def\eea{\end{eqnarray}}
\def\bean{\begin{eqnarray*}}
\def\eean{\end{eqnarray*}}

\def\qquad{{\quad\quad}}

\def\={\, = \, }

\def\Box{\leavevmode\vbox{\hrule
     \hbox{\vrule\kern4pt\vbox{\kern4pt}%
           \vrule}\hrule}}
\def\blackbox{\leavevmode\vrule height 5pt width 4pt depth 0pt\relax}
\catcode`@=11

\def\eqalign#1{\null\,\vcenter{\openup1\jot \m@th
   \ialign{\strut \hfil$\displaystyle{##}$ & $\displaystyle{{}##}$\hfil
      \crcr#1\crcr}}\,}
\def\eqalignrll#1{\null\,\vcenter{\openup1\jot \m@th
   \ialign{\strut \hfil$\displaystyle{##}$ & $\displaystyle{{}##}$\hfil
    & $\displaystyle{{}##}$\hfil
      \crcr#1\crcr}}\,}
\def\eqalignrcl#1{\null\,\vcenter{\openup1\jot \m@th
   \ialign{\strut \hfil$\displaystyle{##}$ &\hfil $\displaystyle{{}##}$\hfil
    & $\displaystyle{{}##}$\hfil
      \crcr#1\crcr}}\,}
\def\eqalignno#1{\displ@y \tabskip\@centering
  \halign to\displaywidth{\hfil$\@lign\displaystyle{##}$\tabskip\z@skip
    &$\@lign\displaystyle{{}##}$\hfil\tabskip\@centering
    &\llap{$\@lign##$}\tabskip\z@skip\crcr
    #1\crcr}}
\newcounter{appendix}
\newcounter{sectionz}
\def\appendix{\advance\c@appendix by 1
   \def\thesectionz {\Alph{appendix}}
\def\thesection{\Alph{appendix}} 
   \ifnum\c@appendix=1 \setcounter{section}{-1} \fi
   \@startsection {section}{1}{\z@}{-3.5ex plus -1ex minus 
  -.2ex}{2.3ex plus .2ex}{\large\bf}}

\catcode`@=12
\newtheorem{lemme}{Lemma}[section]  
\newtheorem{theorem}[lemme]{Theorem}
\newtheorem{corollary}[lemme]{Corollary}
\newtheorem{definition}[lemme]{Definition}
\newtheorem{proposition}[lemme]{Proposition}
\newtheorem{remark}[lemme]{Remark} 

\newtheorem{hypothesis}{Hypothesis}

\def\deblem{\begin{lemme}\it }
\def\finlem{\end{lemme}}
\def\debthm{\begin{theorem}\it }
\def\finthm{\end{theorem}}
\def\debprop{\begin{proposition} \it}
\def\finprop{\end{proposition}}
\def\debcor{\begin{corollary}\it }
\def\fincor{\end{corollary}}
\def\debdef{\begin{definition}\it}
\def\findef{\end{definition}}
\def\debrem{\begin{remark}\em}
\def\finrem{\null\hfill\blackbox\end{remark}}
\newtheorem{thm}{Theorem}  
\newtheorem{lem}[thm]{Lemma}

\def\OO{\mathbb{O}}

\usepackage{amsfonts}
\usepackage{amsbsy}

\usepackage{graphicx}
\usepackage[hang,center]{caption}
\usepackage{verbatim} 
\usepackage{float}
\usepackage{color}

\usepackage{multirow}

\begin{document}

\title[AP-scheme for the anisotropic heat equation]{Highly anisotropic
  temperature balance equation and its asymptotic-preserving
  resolution.}

\author[A. Lozinski, J. Narski, C. Negulescu]{Alexei Lozinski, Jacek Narski, Claudia NEGULESCU}

\address{Universit\'e de Toulouse, UPS, INSA, UT1, UTM, Institut de
  Math\'ematiques de Toulouse, 118 route de Narbonne, F-31062
  Toulouse, France} 
 \email{alexei.lozinski@math.univ-toulouse.fr,  jacek.narski@math.univ-toulouse.fr, claudia.negulescu@math.univ-toulouse.fr}
\date{\today}

\begin{abstract}
This paper deals with the numerical study of a nonlinear, strongly
anisotropic heat equation. The use of standard schemes in this situation leads to poor results, due to the high anisotropy. An Asymptotic-Preserving method is
introduced in this paper, which is second-order accurate in both, temporal and
spacial variables. The discretization in time is done using an L-stable Runge-Kutta scheme. The convergence of the method is shown to be independent of the anisotropy parameter $0 <
\eps <1$, and this for fixed coarse Cartesian grids and for variable
anisotropy directions. The context of this work are magnetically
confined fusion plasmas.
\end{abstract}

.

\maketitle
\section{Introduction} \label{SEC1}
Magnetically confined plasmas are characterized by highly anisotropic
properties induced by the applied strong magnetic
field. Indeed, the charged particles constituting the plasma move rapidly around 
the magnetic field lines, their transverse motion away from the field lines
being constrained by the Lorentz force. In contrast, their
motion along the field lines is relatively unconstrained, so that rather rapid
dynamics along the magnetic fields occurs. This results in an extremely
large ratio of the parallel to the transverse thermal
conductivities, as well as of other parameters characterizing the plasma
evolution. 

A prototype simplified model for the heat diffusion in a magnetically confined plasma can be expressed by the
following nonlinear, degenerate parabolic equation
\be \label{DT}
\partial_t u -  \nabla_{||} \cdot (\kappa_{||}( u) \nabla_{||} u) -\nabla_{\perp} \cdot (\kappa_{\perp} \nabla_{\perp} u)=0\,,
\ee
where the subscripts $||$ (resp. $\perp$) refer to the direction parallel (resp. perpendicular) to the magnetic field
lines and $u$ designates the temperature. In writing out the equation above we have ignored some important physical phenomena coming from convection and turbulence.  
Nevertheless, our equation contains some important features inherited from the full model that lead to substantial difficulties in the numerical treatment of both the full model and our simplified one. The diffusion in the direction perpendicular to the magnetic field lines is usually dominated by the anomalous transport and the corresponding coefficient $\kappa_\perp$ can be taken temperature independent. On the other hand, the coefficient describing the diffusion in the direction parallel to the magnetic field lines, $\kappa_{||}$, is normally much larger and strongly temperature dependent. It can be described by the Spitzer-H\"arm law $\kappa_{||}(u)= \kappa_0 u^{5/2}$ \cite{Wes87}. Moreover, plasma temperatures are extremely high, so that this diffusion coefficient can become very big.  Passing to non-dimensional variables, we shall write therefore the law for $\kappa_{||}$ as 
$$\kappa_{||}(u)= {1\over \eps} u^{5/2},$$ 
where $\eps$ is a small parameter, $0 < \eps \ll 1$. An accurate resolution
of the parallel and perpendicular diffusion processes plays a crucial
role in understanding of the plasma dynamics and the
energy transport phenomena. It is therefore very important to develop and to study efficient
numerical schemes to solve problem  (\ref{DT}). It is also desirable to have a scheme that works robustly for all values of $\eps$
from $\eps<<1$ to $\eps\sim O(1)$ since this parameter enters the equation in combination with a non-linear term so that the effective value of the diffusion coefficient can vary strongly over the computation domain following the variations in $u$. This is the primary motivation of the present work.

Anisotropic, nonlinear diffusion equations of the type (\ref{DT})
arise in several other fields of application and a lot of efforts were made
to construct efficient numerical methods for this challenging
problem. To mention some examples, such non-linear evolution equations
of parabolic type occur in the description of isentropic gas flows
through a porous media \cite{Ashby} or in the
description of transport phenomena in heterogeneous geologic
formations, such as fractured rock systems \cite{Berkowitz}, which are
of fundamental interest for petroleum or groundwater engineering. In
addition, these equations appear also in image processing, related to
the elimination of noise and small-scale details from an image
\cite{Basser, Perona, Weickert} or in the description of the
anisotropic water diffusion in tissues of the nervous system
\cite{Beaulieu}.

From a numerical point of view, problems of the type (\ref{DT}) are
very challenging, as one deals with singularly perturbed
problems, the model changing its type in the limit $\eps \rightarrow 0$. Standard schemes suffer from the presence of very ill conditionned matrices (typically with a condition number of order $1/(\eps h^2)$ where $h$ is the discretization step in space). Solving an equation with such a matrix on a computer accumulates the rounding errors and may lead to completely wrong results. Note that this drawback cannot be overcome by a mesh refinement since it results only in worsening the condition numbers of the matrices in the discretized problem.

Several methods were investigated in literature to cope with this type of
anisotropic problems, using for example high order finite element schemes
\cite{G}, preconditioned conjugate gradient methods in a mixed
spectral/finite difference scheme \cite{L} or introducing an
artificial ``sound'' method, to represent the fast thermal
equilibrium along the field lines \cite{P}. All these methods however are rather involved and moreover their range of application is limited, as they are efficient only until a threshold value for $\eps$, and cannot thus recover the limit regime $\eps \rightarrow 0$. Another class of employed numerical methods are hybrid strategies, which consist in coupling different numerical schemes valid in different regions of the domain. For example in this case, one can couple the resolution of the singular perturbation problem there where $\eps \sim \mathcal{O}(1)$ with the resolution of a limit problem for $\eps \ll 1$. These methods suffer however from the fact, that the coupling conditions between the two models are hard to establish and the interface between the two regions difficult to localize.

The objective of the present paper is to introduce an efficient numerical
scheme based on the Asymptotic-Preserving methodology, which allows for an accurate resolution of the singularly perturbed problem, uniformly in $\eps$, with little additional computational cost, and using a grid which is not necessarily aligned with the magnetic field, so that one can exploit simple Cartesian grids, for example. Initially, AP-techniques were introduced in \cite{ShiJin}, to deal with singularly perturbed kinetic models. The key idea is to reformulate the singularly perturbed problem into an equivalent problem, which is however well-posed if we set $\eps=0$ there. The reformulation of the here proposed method is based, similarly as in \cite{MN}, on introducing a new auxiliary variable, as proposed earlier in an elliptic framework in \cite{DLNN},
and replacing the terms of the equation multiplied by $1/\eps$ are by the new terms with an $O(1)$ factor. From a numerical point of view, this procedure means transforming the problem of condition number $\sim 1/\eps$, into a well-conditioned problem, which switches automatically from the singularly perturbed problem to the limit problem, as $\eps \rightarrow 0$.

The difference between the method
presented in \cite{MN} lies first in the treatment of the
non-linearity. Instead of fixed point iterations used to approach the nonlinearity $(u^{n+1})^{5/2}$, we choose to implement a much simpler linear extrapolation
method (see Section \ref{SEC31} for details). Moreover, we develop here a robust asymptotic-preserving scheme of second order in time, which has no analogue in the existing literature, to the best of our knowledge. Finally, this paper contains also a detailed mathematical study of
the problem.

The paper is organized as follows: Section \ref{SEC2} contains a
description of the problem completed by a mathematical study. In Section
\ref{SEC3}, we present the numerical method based on an asymptotic preserving space discretization and 
develop three different time-discretizations: implicit Euler,
Crank-Nicolson and L-stable Runge-Kutta methods. Finally, in
Section~\ref{sec:numres} we present some numerical results, focusing on the AP-property of the schemes.

\section{Description of the problem and mathematical study} \label{SEC2}

We consider a two or three dimensional anisotropic, nonlinear heat problem, given on a
sufficiently smooth, bounded domain $\Omega \subset \mathbb R^d$,
$d=2,3$ with boundary $\Gamma$. The direction of the
anisotropy is defined by the time-independent vector field $b \in
(C^{\infty}(\Omega))^d$, satisfying $|b(x)|=1$ for all $x \in \Omega$.

\noindent Given this vector field $b$, one can decompose now vectors
$v \in \mathbb R^d$, gradients $\nabla \phi$, with $\phi(x)$ a scalar
function, and divergences $\nabla \cdot v$, with $v(x)$ a vector
field, into a part parallel to the anisotropy direction and a part
perpendicular to it.  These parts are defined as follows:
\begin{equation} 
  \begin{array}{llll}
    \ds v_{||}:= (v \cdot b) b \,, & \ds v_{\perp}:= (Id- b \otimes b) v\,, &\textrm{such that}&\ds
    v=v_{||}+v_{\perp}\,,\\[3mm]
    \ds \nabla_{||} \phi:= (b \cdot \nabla \phi) b \,, & \ds
    \nabla_{\perp} \phi:= (Id- b \otimes b) \nabla \phi\,, &\textrm{such that}&\ds
    \nabla \phi=\nabla_{||}\phi+\nabla_{\perp}\phi\,,\\[3mm]
    \ds \nabla_{||} \cdot v:= \nabla \cdot v_{||}  \,, & \ds
    \nabla_{\perp} \cdot v:= \nabla \cdot v_{\perp}\,, &\textrm{such that}&\ds
    \nabla \cdot v=\nabla_{||}\cdot v+\nabla_{\perp}\cdot v\,,
  \end{array}
\end{equation}
where we denoted by $\otimes$ the vector tensor product.

The boundary $\Gamma$ can be decomposed into three components following the sign of the
intersection with $b$: 
$$
\Gamma_{||}:= \{ x \in \Gamma \,\, / \,\, b(x) \cdot n(x) =0 \}\,, 
$$
$$
\Gamma_{in}:= \{ x \in \Gamma \,\, / \,\, b(x) \cdot n(x) <0 \}\,, \quad \Gamma_{out}:= \{ x \in \Gamma \,\, / \,\, b(x) \cdot n(x) > 0 \}\,, 
$$
and $\Gamma_{\perp}=\Gamma_{in}\cup \Gamma_{out}$. The vector $n$ is here the unit
outward normal on $\Gamma$.

With these notations we can now introduce the mathematical problem, we
are interested to study. We are searching for the particle (ions or
electrons) temperature $u(t,x)$, solution of the evolution equation
\be \label{P}
(P) \,\,\,
\left\{
\begin{array}{l}
\partial_t u - {1 \over \eps} \nabla_{||} \cdot (A_{||} u^{5/2} \nabla_{||} u) -\nabla_{\perp} \cdot (A_{\perp} \nabla_{\perp} u)=0\,, \quad \textrm{in} \quad [0,T] \times \Omega\,,\\[3mm]
{1 \over \eps} n_{||} \cdot (A_{||} u^{5/2}(t,\cdot) \nabla_{||} u(t,\cdot)) + n_{\perp} \cdot (A_{\perp} \nabla_{\perp} u(t,\cdot))= -\gamma \,  u(t,\cdot)\,, \quad \textrm{on} \quad [0,T] \times \Gamma_{\perp}\,, \\[3mm]
\nabla_{\perp} u (t,\cdot) =0 \,, \quad \textrm{on} \quad [0,T] \times \Gamma_{||}\,, \\[3mm]
u(0,\cdot)=u^{0} (\cdot)\,, \quad \textrm{in} \quad \Omega\,.
\end{array}
\right.
\ee
The coefficient $\gamma$ is zero for
electrons and $\gamma >0$ for ions \cite{Tam07, Wes87}. The problem (\ref{P}) describes
the diffusion of an initial temperature $u^{0}$ within the time interval
$[0,T]$ and its outflow through the boundary $\Gamma_{\perp}$. Let us denote in the
following the time-space cylinder by $Q_T:=(0,T)\times \Omega$. The
parameter $0<\eps\ll 1$ can
be very small and is responsible for the high anisotropy of the
problem.  
We shall suppose all along this paper, that the coefficients $A_{||}$ and $A_{\perp}$ are of the same order of magnitude, satisfying 

\begin{hypothesis}
\label{hypo} Let $\Gamma _{\perp }$ consist of two connected components $%
\Gamma _{in}=\{x\in \Gamma /n\cdot b<0\}$ and $\Gamma _{out}=\{x\in \Gamma
/n\cdot b>0\}$ such that\\
either

\textbf{case A:} $n=-b$ on $\Gamma _{in}$ and $n=b$ on $\Gamma _{out}$\\
or

\textbf{case B:} $n\cdot b>-\varkappa $ on $\Gamma _{in}$ and $n\cdot
b<\varkappa $ on $\Gamma _{out}$ with some constant $0<\varkappa <1.$\\
All the components $\Gamma _{in}$, $\Gamma _{out}$ and $\Gamma _{||}$ are sufficiently smooth.
We suppose moreover $0<\eps\leq 1$ and $\gamma \geq 0$ fixed. The diffusion
coefficients $A_{\parallel }\in W^{1,\infty}(\bar\Omega )$ and $A_{\perp }\in 
\mathbb{M}_{d\times d}(W^{1,\infty}(\bar\Omega ))$ are supposed to satisfy 
\begin{gather}
0<A_{0}\leq A_{\parallel }(x)\leq A_{1}\,,\quad \text{f.a.a.}\,\,\,x\in
\Omega ,  \label{eq:J48a1} \\[3mm]
A_{0}||v||^{2}\leq v^{t}A_{\perp }(x)v\leq A_{1}||v||^{2}\,,\quad \forall
v\in \mathbb{R}^{d}\,\,\,\text{and}\,\,\,\text{f.a.a.}\,\,\,x\in \Omega ,
\label{eq:J48a3}
\end{gather}%
with $0<A_{0}<A_{1}$ some constants.
\end{hypothesis}

Putting formally
$\eps =0$ in (\ref{P}) leads to the following ill-posed problem, admitting
infinitely many solutions 
\be \label{R}
\left\{
\begin{array}{l}
 - \nabla_{||} \cdot (A_{||} u^{5/2} \nabla_{||} u) =0\,, \quad \textrm{in} \quad [0,T] \times \Omega\,,\\[3mm]
n_{||} \cdot (A_{||} u^{5/2}(t,\cdot) \nabla_{||} u(t,\cdot)) = 0\,, \quad \textrm{on} \quad [0,T] \times \Gamma_{\perp}\,, \\[3mm]
\nabla_{\perp} u (t,\cdot) =0 \,, \quad \textrm{on} \quad [0,T] \times \Gamma_{||}\,, \\[3mm]
u(0,\cdot)=u^{0} (\cdot)\,, \quad \textrm{in} \quad \Omega\,.
\end{array}
\right.
\ee
Indeed, all functions which are constant along the field lines, meaning $\nabla_{||} u \equiv 0$, and satisfying moreover the boundary condition on $\Gamma_{||}$, are solutions of this problem. From a
numerical point of view, this ill-posedness in the limit $\eps \rightarrow 0$ can be detected by the fact, that trying to solve (\ref{P}) with standard
schemes leads to a linear system, which is very ill-conditioned for
$0<\eps\ll 1$, in particular with a condition  number of the order of
$1/ \eps$.  

The aim of this paper will be to introduce an efficient numerical
method, permitting to solve (\ref{P}) accurately on a coarse Cartesian
grid, which has not to be adapted to the field lines of $b$ and whose
mesh size is independent of the value of $\eps$. The here proposed
scheme belongs to the category of Asymptotic-Preserving schemes,
meaning they are stable independently of the small parameter $\eps$
and consistent with the limit problem, if $\eps$ tends to zero. The
construction of the here developed AP-scheme is an adaptation of a
method introduced by the authors in an elliptic framework (see
\cite{DLNN}), to the here considered non-linear and time-dependent
problem, and is based on a reformulation of the singularly perturbed problem (\ref{P}) into an equivalent problem, which appears to be well-posed in the limit $\eps \rightarrow 0$. But before introducing the AP-approach, we will start by studying in the following section the mathematical properties of problem (\ref{P}). The test configuration chosen all along this paper is the diffusion of an initial temperature hot spot (see
Figure \ref{ima1}) along arbitrary magnetic field lines $b$.
\begin{figure}[h]
\begin{center}
    \includegraphics{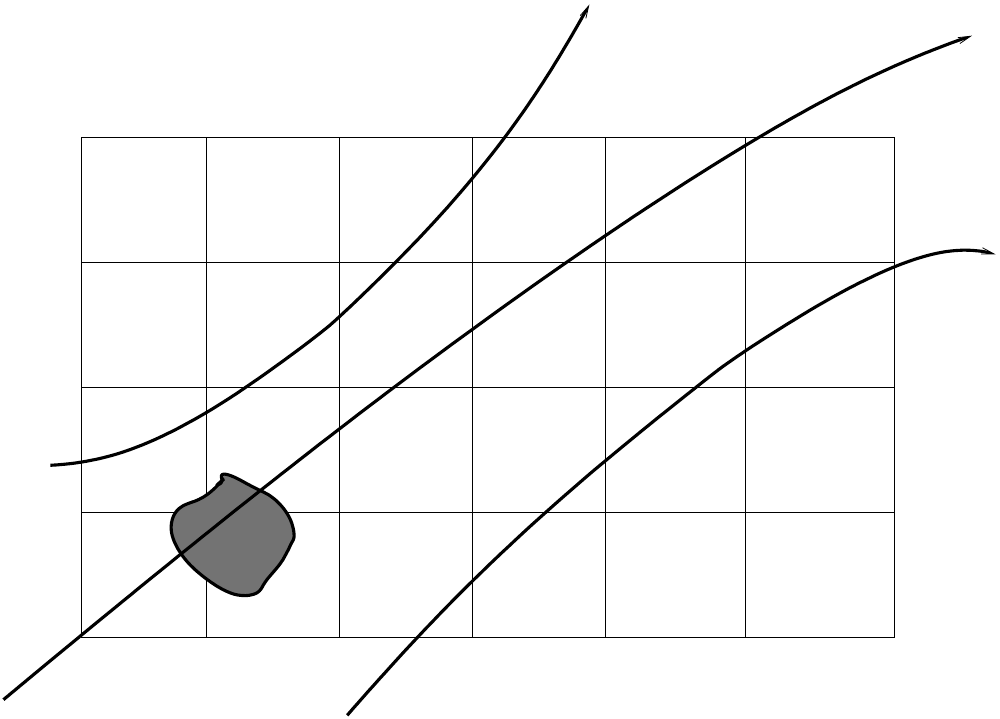}
\end{center}
\caption{\label{ima1} Diffusion of a hot temperature spot along the magnetic field lines.}
\end{figure}

\subsection{Mathematical properties} \label{SEC21}
Before starting with the presentation of the AP numerical scheme, let us first check some
properties of the diffusion problem (\ref{P}) for fixed $\eps >0$. 
For notational simplicity, we consider a slightly more general form of problem (P)%
\begin{equation}
(P_{m})\,\,\,\left\{ 
\begin{array}{l}
\partial _{t}u-\nabla _{||}\cdot (A_{||}|u|^{m-1}\nabla _{||}u)-\nabla
_{\perp }\cdot (A_{\perp }\nabla _{\perp }u)=0\,,\quad \text{in}\quad
\lbrack 0,T]\times \Omega \,, \\[3mm] 
A_{||}|u|^{m-1}n_{||}\cdot \nabla _{||}u+A_{\perp }n_{\perp }\cdot \nabla
_{\perp }u=-\gamma \,u\,,\quad \text{on}\quad \lbrack 0,T]\times \Gamma
_{\perp }\,, \\[3mm] 
\nabla _{\perp }u=0\,,\quad \text{on}\quad \lbrack 0,T]\times \Gamma _{||}\,,
\\[3mm] 
u(0,\cdot )=u^{0}(\cdot )\,,\quad \text{in}\quad \Omega \,,%
\end{array}%
\right.  \label{Pm}
\end{equation}%
for any $m\geq 1$. We obtain the particular case (\ref{P}) by setting $m=5/2+1$
and redefining $A_{||}$ as $\frac{1}{\varepsilon }A_{||}$ for any $%
\varepsilon >0$. Equations of the type (\ref{Pm}) are rather well
studied in the literature. We refer to the classical works
\cite{dubinsk64, dubinsk65, lions69} as well as to the more modern
literature on ``The porous medium equation" as reviewed in
\cite{Arons86, Vazquez}. However, all these references normally treat
only an isotropic version of the problem above, i.e. the non-linearity
of the type $u^{m-1}$ is present in front of all the derivatives of
$u$. An anisotropic equation of the form (\ref{Pm}) is studied in
\cite{jian06}, but only in the case $m<\frac{d+1}{d-1}$, so that the value of $m$
pertinent to our application is not covered. Another feature of our
setting, which is not sufficiently covered in the existing literature,
is the prescription of Robin boundary conditions. This is the reason
why we wish to study in this paper the existence, uniqueness and
positivity of solutions to (\ref{Pm}). 

We shall first introduce the concept of weak solution of
problem (\ref{Pm}) and state the existence/uniqueness theorem. Note
that unlike the literature cited above, we assume from the beginning
that the initial conditions are bounded and strictly positive, and
prove the same properties for the weak solutions. Our treatment is
thus performed under much less general assumptions than usually
required, but this is quite enough for our application.

\begin{definition} \textbf{(Weak solution)}
Let $u^{0}\in L^{\infty }(\Omega )$. Then $u\in \mathcal{W}$ with 
\begin{align*}
\mathcal{W}:=\{& u\in L^{\infty }(Q_{\infty}),\text{ such that }\forall \, T>0 
\\
&\,
\nabla _{\perp }u\in
L^{2}(Q_{T})\,,\quad |u|^{m-1}\nabla _{||}u\in L^{2}(Q_{T})\,,\quad \partial
_{t}u\in L^{2}(0,T;(H^{1}(\Omega ))^{\ast })\}\,,
\end{align*}%
is called a weak solution of problem (\ref{Pm}), if $u(0,\cdot )=u^{0}$ and if for all $T>0$ one has
\begin{eqnarray}
&&\int_{0}^{T}\langle \partial _{t}u(t,\cdot ),\phi (t,\cdot )\rangle
_{(H^{1})^{\ast },H^{1}}\,dt+\int_{0}^{T}\int_{\Omega }A_{||}|u|^{m-1}\nabla
_{||}u\cdot \nabla _{||}\phi\, dxdt  \label{weak} \\
&&+\int_{0}^{T}\int_{\Omega }A_{\perp }\nabla _{\perp }u\cdot \nabla _{\perp
}\phi\, dxdt+\gamma \int_{0}^{T}\int_{\Gamma _{\perp }}u\phi \,d\sigma
\,dt=0,\quad \forall \phi \in \mathcal{D}  \notag
\end{eqnarray}%
where $\mathcal{D}=L^{2}(0,T;H^{1}(\Omega ))$.\newline
\end{definition}

\begin{remark}\label{remTh}
All the terms in this variational formulation are well-defined for any $u\in 
\mathcal{W}$ and $\phi \in \mathcal{D}$. Indeed, as $u\in L^{\infty }(Q_{T})$, $\nabla _{\perp}u\in L^{2}(Q_{T})$
and $|u|^{m-1}\nabla _{||}u\in L^{2}(Q_{T})$, one also has $|u|^{m-1}\nabla
u\in L^{2}(Q_{T})$, and thus $|u|^{m}\in
L^{2}(0,T;H^{1}(\Omega ))$. This means that $|u|^{m}$ has a trace on $%
\partial \Omega $, belonging to $L^{2}(\partial \Omega )$. As $L^{2m}(\partial \Omega) \subset L^{2}(\partial \Omega)$ 
for all $m \ge 1$ ($\Omega$ is bounded), one has $u_{| \partial \Omega} \in L^{2}(\partial \Omega )$, justifying thus the boundary integral in (\ref{weak}). 
Moreover, we have the continuous inclusion $\mathcal{W}\subset C([0,T];(H^1(\Omega))^*)$. This shows that one can impose the initial condition $u(0,\cdot )=u^{0}$ with $u^{0} \in L^\infty(\Omega) \subset (H^1(\Omega))^*$.
\end{remark}

\begin{remark}\label{remTh1}
Actually, we have a sharper characterization of continuity in time for functions in $\mathcal{W}$. Indeed, let $Y$ be a Banach space and let us denote by $C([0,T];Y_w)$ the space of weakly continuous functions with values in $Y$, which means for each $\psi \in Y^*$ the mapping $t \longmapsto \langle \psi, u (t)\rangle_{Y^*,Y}$ is continuous. Then, one can prove (see \cite{LM} for more details) that
$\mathcal{W} \subset L^{\infty}(0,T;L^2(\Omega)) \cap C([0,T];(H^1(\Omega))^*) \subset C([0,T];L^{2}_w(\Omega))$.
\end{remark}

\begin{theorem}\label{mainTh} {\bf (Existence/Uniqueness/Positivity)}
Let $u^{0}\in L^{\infty }(\Omega )$ satisfy $0<\beta\leq u^{0}\leq M<\infty $ on $%
\Omega $, for some $\beta >0$. Assume Hypothesis \ref{hypo} and $m\geq 1$. Then there exists a unique
weak solution $u\in \mathcal{W}$ of (\ref{Pm}), which satisfies $c e^{-Kt} \leq u\leq
M$ a.e. on $Q_{\infty}$, with some sufficiently small $c>0$ and some sufficiently large $K>0$.
\end{theorem}

Before proving this theorem, let us also define the sub- and super-solutions to problem (\ref{Pm}) and establish a comparison principle for them. 

\begin{definition} \textbf{(Sub/super-solutions)}
A function $u\in \mathcal{W}$ is called a weak sub- (resp. super-) solution to problem (\ref{Pm}) if the variational formulation (\ref{weak}) is verified 
for all $\phi \in \mathcal{D}$ with $\phi\ge 0$ on $Q_\infty$, and where the equality sign is replaced by $\le$ (resp. $\ge$).
\end{definition}

\begin{lem}\label{compPr} {\bf (Comparison principle)}
Assume Hypothesis \ref{hypo} and $m\geq 1$. Let $u_{1}$ be a non-negative
sub-solution and $u_{2}$ be a non-negative super-solution to (\ref{Pm}) such
that $u_{1}(0,x)\leq u_{2}(0,x)$ for a.a. $x\in \Omega $. If at least one of
the functions $u_{1}$, $u_{2}$ is strictly positive, i.e. $\forall T>0$ $%
\exists \beta _{T}>0$ such that $u_{1}\geq \beta _{T}$ or $u_{2}\geq \beta _{T}$ on $Q_{T}$, 
then $u_{1}\leq u_{2}$ on $Q_{\infty }$.
\end{lem}

\begin{proof}
For any $k>0$, introduce the function $H_{k}:\mathbb{R}\rightarrow 
\mathbb{R}$ as 
\begin{equation*}
H_{k}(u)=\left\{ 
\begin{array}{ll}
0, & \text{if }u\leq 0 \\ 
ku, & \text{if }0<u\leq \frac{1}{k} \\ 
1, & \text{if }u>\frac{1}{k}%
\end{array}%
\right. 
\end{equation*}%
and put $\phi =H_{k}(u_{1}^{m}-u_{2}^{m})$. Note that $\phi \in
L^{2}(0,T;H^{1}(\Omega ))$ since $u_{1},u_{2}\in \mathcal{W}$ and thus $%
\nabla u_{1}^{m},\nabla u_{2}^{m}\in L^{2}(Q_{T})$. Observe also that the
gradient of $\phi $ is zero outside from the set $\omega _{T}^{k}=\{(t,x)\in 
\bar{Q}_{T}:0<u_{1}^{m}-u_{2}^{m}<\frac{1}{k}\}$. Choosing this $\phi $ as
the test function in the inequalities (\ref{weak}) for $u_{1}$ and $u_{2}$
and subtracting the second one form the first one gives%
\begin{align}
& \int_{0}^{T}\langle \partial
_{t}(u_{1}-u_{2}),H_{k}(u_{1}^{m}-u_{2}^{m})\rangle _{(H^{1})^{\ast
},H^{1}}\,dt\leq -\frac{k}{m}\iint\limits_{\omega _{T}^{k}}A_{||}|\nabla
_{||}(u_{1}^{m}-u_{2}^{m})|^{2}dxdt  \label{inqcmp} \\
& {\quad \quad }-mk\iint\limits_{\omega _{T}^{k}}A_{\perp }\nabla _{\perp
}(u_{1}-u_{2})\cdot \lbrack u_{1}^{m-1}\nabla _{\perp
}(u_{1}-u_{2})+(u_{1}^{m-1}-u_{2}^{m-1})\nabla _{\perp }u_{2}]dxdt  \notag \\
& {\quad \quad }-\gamma \int_{0}^{T}\int_{\Gamma _{\perp
}}(u_{1}-u_{2})H_{k}(u_{1}^{m}-u_{2}^{m})\,d\sigma dt  \notag \\
& \qquad\qquad
\leq mk \iint\limits_{\omega _{T}^{k}}A_{\perp
}|\nabla _{\perp }(u_{1}-u_{2})\cdot \nabla _{\perp }u_{2}|
(u_1^{m-1}-u_2^{m-1})dxdt
\notag
\end{align}
since $(u_{1}-u_{2})$ and $H_{k}(u_{1}^{m}-u_{2}^{m})$ are of the same sign.
We have moreover 
$$
u_{1}^{m-1}-u_{2}^{m-1}
\leq C_{m}\frac{u_{1}^{m}-u_{2}^{m}}{u_{1}+u_{2}}
\leq \frac {C_{m}}{k\beta}
$$ 
on $\omega _{T}^{k}$ with a constant $C_{m}>0$ depending only on $m$. 
Indeed, the first inequlity here holds for any $u_{1} \ge u_{2}>0$  and the second inequality follows by noting that 
$u_{1}^{m}-u_{2}^{m}<1/k$ on $\omega _{T}^{k }$ and $u_{1}+u_{2}\geq
\beta _{T}$ on $Q_{T}$. We see that (\ref{inqcmp}) combined with the inequality above implies
$$
\int_{0}^{T}\langle \partial
_{t}(u_{1}-u_{2}),H_{k}(u_{1}^{m}-u_{2}^{m})\rangle _{(H^{1})^{\ast
},H^{1}}\,dt
\leq 
\frac{mC_{m}}{\beta _{T}}\iint\limits_{\omega _{T}^{k}}A_{\perp
}|\nabla _{\perp }(u_{1}-u_{2})\cdot \nabla _{\perp }u_{2}|dxdt .
$$
Let us now take the limit $k\rightarrow \infty $ in this inequality. We
have $meas(\omega _{T}^{k})\rightarrow 0$ so that
\begin{equation}\label{limk1}
\limsup_{k\to\infty}\int_{0}^{T}\langle \partial_{t}(u_{1}-u_{2}),H_{k}(u_{1}^{m}-u_{2}^{m})\rangle _{(H^{1})^{\ast
},H^{1}}\,dt
\leq 0.
\end{equation}
On the other hand,
\begin{equation}\label{limk2}
\lim_{k\to\infty}\int_{0}^{T}\langle \partial_{t}(u_{1}-u_{2}),H_{k}(u_{1}^{m}-u_{2}^{m})\rangle _{(H^{1})^{\ast
},H^{1}}\,dt
=
\int_{\Omega }(u_{1}-u_{2})^{+}(T,x)dx-\int_{\Omega
}(u_{1}-u_{2})^{+}(0,x)dx. 
\end{equation}
where $(u)^{+}=(u+|u|)/2$ denotes the positive part of $u$.
Indeed, $H_{k}(u_{1}^{m}-u_{2}^{m})\rightarrow H(u_{1}-u_{2})$ a.a. on $Q_{T}$ where $%
H$ denotes the Heaviside function ($H(x)=1$ for $x>0$ and $H(x)=0$ for $%
x\leq 0$). Observing that $\partial _{t}(u_{1}-u_{2})^{+}=\partial
_{t}(u_{1}-u_{2})H(u_{1}-u_{2})$ in the sense of distributions, proves (\ref{limk2}) for sufficiently smooth $u_1,u_2$.  
A standard density argument shows then that (\ref{limk2}) actually holds for any $u_1,u_2\in\mathcal{W}$. Note, in particular, that the terms at the right-hand side of (\ref{limk2}) are well defined for functions in $\mathcal{W}$ thanks to the inclusion
$\mathcal{W} \subset C([0,T];L^{2}_w(\Omega))$, cf. Remark \ref{remTh1}. 
Comparing (\ref{limk1}) and (\ref{limk2}) and taking into account $(u_{1}-u_{2})^{+}=0$ on $\Omega $ at $t=0$, yields
\begin{equation}
\int_{\Omega }(u_{1}-u_{2})^{+}(T,x)dx\leq 0,  \label{concmp}
\end{equation}%
which implies $u_{1}\leq u_{2}$ on $Q_{\infty}$.
\end{proof}

The construction of the following remarkable sub-solution is essentially due to M. Pierre \cite{MPierre}. 

\begin{lem} {\bf (Construction of a weak solution)}\label{subsol}
Assume Hypothesis \ref{hypo} and $m\geq 1$. For any $\beta >0$, there exists a weak sub-solution $w$ to problem (\ref{Pm}) satisfying 
$c \le w(0,x)\leq \beta $ for $x\in \Omega $ and $w(t,x)\geq ce^{-Kt}$ for $%
(t,x)\in Q_{\infty }$, with some constants $c,K>0$ which depend only on $%
\beta $.
\end{lem}

\begin{proof}
We will construct a smooth sub-solution $w$ satisfying all the announced
properties. We thus rewrite the definition of a sub-solution in the strong
form supposing from the beginning that $w\geq 0$:%
\begin{eqnarray}
\partial _{t}w-\frac{1}{m}\nabla _{||}\cdot (A_{||}\nabla _{||}w^{m})-\nabla
_{\perp }\cdot (A_{\perp }\nabla _{\perp }w) &\leq &0,\quad \text{ on }\quad (0,\infty)
\times \Omega \,  \label{eqw} \\
{1 \over m}A_{||}n_{||}\cdot \nabla _{||}w^{m}+A_{\perp }n_{\perp }\cdot \nabla _{\perp
}w+\gamma w &\leq &0,\quad \text{ on }\quad (0,\infty)\times \Gamma
_{\perp },  \label{eqbw1}\\
n_{\perp }\cdot \nabla _{\perp }w&\leq& 0,\quad \text{ on }%
\quad (0,\infty)\times \Gamma _{||}\,  \label{eqbw2}
\end{eqnarray}
The construction of such a function $w$ will be performed separately for the two cases mentioned in Hypothesis \ref{hypo}.

\noindent\textbf{Case A:}
One can construct in this case a new coordinate system $\xi _{1},\ldots ,\xi _{d}$ on $\Omega$ such that the coordinate lines $\xi_d$ coincide with the $b$-field lines and the surfaces $\xi_d=const$ are perpendicular to these lines. Domain $\Omega$ is represented in these coordinates by a cylinder $\Omega_\xi=D\times(0,1)$ with $\xi'=(\xi _{1},\ldots ,\xi _{d-1})\in D$ and $\xi_d\in(0,1)$. We thus have $\nabla_{||}=b\chi\frac{\partial }{\partial \xi _{d}}$ with some scalar strictly positive field $\chi$. We assume that the component $\Gamma_{in}$ of the boundary is represented  by $D\times\{\xi_d=0\}$, $\Gamma_{out}$ is represented  by $D\times\{\xi_d=1\}$ and $\Gamma_{||}$ is represented  by $\partial D\times(0,1)$. 

We are searching now for a sub-solution under the form $w(t,x)=\delta(t)(\sin (\pi
\xi _{d})+\eta(t))^{1/m}$ where $\delta(t)$ and $\eta(t)$ are two positive decreasing
functions which are yet to be adjusted. We observe immediately that $\nabla
_{\perp }w=0$ on $\Omega $ for all time so that (\ref{eqbw2}) is
automatically satisfied. The remaining boundary conditions (\ref{eqbw1}) \
should be checked on $\Gamma _{in}$ and $\Gamma _{out}$. We remind that $%
n=n_{||}=b$ on $\Gamma _{out}$ ($\xi_{d}=1$). Similarly, $n=n_{||}=-b$ on $\Gamma
_{in} $ ($\xi _{d}=0$). Substituting the Ansatz for $w$ into (\ref{eqbw1}) now gives%
\begin{equation*}
-{\frac{1}{m}}A_{||}\chi \delta^{m}\pi +\gamma a\eta^{\frac{1}{m}}\leq 0,\text{ for }%
\xi _{d}=0\text{ and }\xi _{d}=1.
\end{equation*}%
This holds if one takes $\eta=K_{1}\delta^{m(m-1)}$ where $K_{1}=\left( \min_{\xi
\in \bar{\Omega}_{\xi }}{\frac{\pi }{m\gamma }}A_{||}\chi \right) ^m >0$.

It remains to check (\ref{eqw}). Substituting the Ansatz for $w$, this
inequality is reduced to
\begin{equation}\label{adoti}
\dot{\delta}(\sin (\pi \xi _{d})+\eta)^{\frac{1}{m}}-\frac{\delta^{m}}{m}\chi 
\frac{\partial }{\partial \xi _{d}}\left( A_{||}\chi \right) \pi \cos (\pi \xi
_{d})+\frac{\delta^{m}}{m}A_{||}\chi ^{2}\pi ^{2}\sin (\pi \xi _{d})\leq 0.
\end{equation}%
Note that we have denoted the time derivative here by a dot and we neglected
a term with $\dot{\eta}$ since it is negative (the function $\eta(t)$ is
decreasing). We divide now both sides by $(\sin (\pi \xi _{d})+\eta)^{\frac{1}{m}}$ and bound each term on the left-hand side as
\begin{equation*}
-\frac{\delta^{m} \pi \cos (\pi \xi _{d}) \chi}{m(\sin (\pi \xi _{d})+\eta)^{\frac{1}{m}}}
\frac{\partial (A_{||}\chi) }{\partial \xi _{d}}
\leq 
\frac{\pi\delta^{m}\chi}{m\eta^{\frac{1}{m}}}
\left|\frac{\partial (A_{||}\chi) }{\partial \xi _{d}} \right|
= \frac{\pi\delta^{m}\chi}{mK_{1}^{\frac{1}{m}}\delta^{m-1}}
\left|\frac{\partial (A_{||}\chi) }{\partial \xi _{d}} \right|
= \delta \frac{\pi}{mK_{1}^{\frac{1}{m}}} \chi
\left|\frac{\partial (A_{||}\chi) }{\partial \xi _{d}} \right|
\end{equation*}%
and%
\begin{equation*}
\frac{\delta^{m}}{m}A_{||}\chi ^{2}\pi ^{2} \frac{\sin (\pi \xi _{d})}{(\sin (\pi \xi _{d})+\eta)^{\frac{1}{m}}}
\leq
\frac{\delta^{m}}{m}A_{||}\chi ^{2}\pi ^{2}  (\sin(\pi \xi _{d}))^{1-\frac{1}{m}}
\leq \frac{\delta^{m}}{m}A_{||}\chi ^{2}\pi ^{2}. 
\end{equation*}%
We see now that inequality (\ref{adoti}) will be satisfied if we require%
\begin{equation}\label{eqdifa}
\dot{\delta}+K_{2}\delta+K_{3}\delta^{m}\leq 0
\end{equation}%
with%
\begin{equation*}
K_{2}=\frac{\pi }{mK_{1}^{\frac{1}{m}}}\max_{\xi \in \Omega _{\xi }}\left\vert \chi \frac{%
\partial }{\partial \xi _{d}}\left( A_{||}\chi \right) \right\vert \text{
and }K_{3}=\frac{\pi ^{2}}{m}\max_{\xi \in \Omega _{\xi }}\left\vert
A_{||}\chi ^{2}\right\vert .
\end{equation*}%
One can thus take $\delta(t)=\delta_{0}e^{-(K_{2}+K_{3})t}$ with any $\delta_{0}\in
(0,1] $.

In summary, $w(t,x)=\delta_{0}e^{-(K_{2}+K_{3})t}(\sin(\pi\xi_d)
+K_1a_{0}^{m(m-1)}e^{-m(m-1)(K_{2}+K_{3})t})^{\frac{1}{m}}$ is a
sub-solution. Clearly, for any $\beta >0$ one can take $\delta_{0}$ small enough
so that $w(0,x) \leq \beta $.
Moreover, for any $t$, $w(t,x)\geq \delta_{0}^{m} K_1^{1/m} e^{-m(K_{2}+K_{3})t}$ so that we
have proved the statement of the Lemma putting $c=\delta_{0}^{m} K_1^{1/m}$, $K=m(K_{2}+K_{3})$.

\noindent\textbf{Case B:}
Let $\phi \in C^{2}(\bar{\Omega})$ be a strictly positive function such that $\phi (x)\geq 1$ on $\bar\Omega $, $n_{\perp }\cdot \nabla _{\perp }\phi =n\cdot \nabla \phi
\leq 0$ on $\Gamma _{||}$ and $\nabla\phi=-\zeta\phi n$ on $\Gamma _{\perp }$ with some sufficiently big constant $\zeta>0$, to be prescribed later. 

We are searching now for a sub-solution under the form $w(t,x)=ce^{-Kt}\phi (x)$ where $c,K$ are some positive constants which are
yet to be adjusted. We observe immediately that (\ref{eqbw2}) is
automatically satisfied for such $w$. The left-hand side of (\ref{eqbw1})
can be written as%
\begin{eqnarray*}
&&
{\frac{c^{m}}{m}}e^{-mKt}A_{||}n_{||}\cdot \nabla _{||}\phi
^{m}+ce^{-Kt}(A_{\perp }n_{\perp }\cdot \nabla _{\perp }\phi +\gamma \phi
) 
\\
&&\qquad
\leq ce^{-Kt}\left(
-\zeta c^{m-1}e^{-(m-1)Kt} A_{||}|n_{||}|^2\phi ^{m}
-\zeta A_{\perp}n_{\perp}\cdot n_{\perp} \phi
+\gamma \phi 
\right)
\\
&&\qquad
\leq ce^{-Kt} (-\zeta A_{\perp}n_{\perp}\cdot n_{\perp}  +\gamma \phi )
\end{eqnarray*}%
and thus it is negative provided $\zeta$ is chosen sufficiently big. Indeed, $A_{\perp }n_{\perp }\cdot n_{\perp }$ 
is uniformly bounded from below by a positive constant in view of the geometrical hypothesis of case B. 

It remains to check (\ref{eqw}). Substituting the Ansatz for $w$
into this inequality yields
\begin{equation}
ce^{-Kt}\left( -K\phi -{\frac{c^{m-1}e^{-(m-1)Kt}}{m}}\nabla _{||}\cdot
(A_{||}\nabla _{||}\phi ^{m})-\nabla _{\perp }\cdot (A_{\perp }\nabla
_{\perp }\phi )\right) \leq 0.  
\end{equation}%
This inequality is satisfied provided we take $c\leq 1$ and 
\begin{equation*}
K= 
\frac{1}{m} \max_{x\in \Omega }\left\vert \nabla _{||}\cdot (A_{||}\nabla
_{||}\phi ^{m})
)\right\vert 
+\max_{x\in \Omega }\left\vert
\nabla _{\perp }\cdot (A_{\perp }\nabla _{\perp }\phi)
\right\vert
.
\end{equation*}%
Finally, for any $\beta >0$ one can take $c$ small enough so that $%
w(0,x)=c\phi (x)\leq \beta $. Lemma is thus proved also in case B.

\end{proof}

\begin{remark}
In the case of a simple ``aligned" geometry, {\it i.e.} $b=e_{d}$ and $%
\Omega =D\times ]0,L[$ with $D$ a domain in $\mathbb{R}^{d-1}$, and supposing $A_{||}=const$, one can easily construct a sub-solution satisfying a sharper estimate: 
under the assumptions of the preceding Lemma, there is a sub-solution such that 
$$
w(t,x)\geq \frac{C}{(1+Kt)^{\frac{m}{m-1}}}. 
$$
Indeed, one can repeat the proof as in case A of the preceding Lemma, taking $\xi'=(x_1,\ldots,x_{d-1})$, $\xi_d=x_d/L$, up to the differential inequality (\ref{eqdifa}). 
One observes now that $K_2=0$ so that one can take $\delta(t)=\frac{\delta_0 }{(1+Kt)^{\frac{1}{m-1}}}$ with any $\delta_0 >0$ and $K=(m-1)K_{3}\delta_0 ^{m-1}$.
Our sub-solution is thus $w=\frac{\delta_0 }{(1+Kt)^\frac{1}{m-1}}\left( \sin(\pi x_d) +
\frac{K_1 \delta_0 ^{m(m-1)}}{(1+Kt)^m}\right) ^{\frac{1}{m}}$  and $w\geq \frac{K_1^{\frac 1m} \delta_0 ^{m}}
{(1+Kt)^{\frac m{m-1}} }$ as stated.
\end{remark}

Let us now turn to the proof of our main result.
\begin{proof}[{\bf Proof of Theorem \ref{mainTh}}]
We shall first regularize the problem, in order to avoid the degeneracy. Then, in a second
step, we shall treat the non-linearity via a fixed point argument. Finally, a
priori estimates shall help us to pass to the limit in the regularized
problem, to prove existence. The comparison principle above will be
used to establish the uniqueness and the positivity of the
solution. Let us now detail these steps. 

\underline{1st step: Regularization}\\
Fix $0<\alpha <1$ and assume that $M>0$ is an upper bound for $u^{0}$.
Introduce for notational simplicity the following functions $a_{\alpha },\Lambda _{\alpha }:\mathbb{R}%
\rightarrow \mathbb{R}^{+}$ %
\begin{equation*}
a_{\alpha }(u):=[\alpha +\min (|u|,M)]^{m-1}\,,\quad \Lambda _{\alpha
}(u):=\int_{0}^{u}a_{\alpha }(s)\,ds\,,
\end{equation*}%
and consider the regularized version of (\ref{weak}): find $u_{\alpha }\in
W_{2}^{1}(0,T;H^{1}(\Omega ),L^{2}(\Omega ))$ (i.e. $u_{\alpha }\in
L^{2}(0,T;H^{1}(\Omega ))$ and $\partial _{t}u_{\alpha }\in
L^{2}(0,T;(H^{1}(\Omega ))^{\ast }))$ such that $u_{\alpha }(0,\cdot )=u^{0}$
and 
\begin{equation}
\begin{array}{lll}
\label{var} &  & \ds\int_{0}^{T}\langle \partial _{t}u_{\alpha }(t,\cdot
),\phi (t,\cdot )\rangle _{(H^{1})^{\ast
},H^{1}}\,dt+\int_{0}^{T}\int_{\Omega }A_{||}a_{\alpha }(u_{\alpha })\nabla
_{||}u_{\alpha }\cdot \nabla _{||}\phi\, dxdt \\[3mm]
&  & \hspace{1cm} \ds+\int_{0}^{T}\int_{\Omega }A_{\perp }\nabla _{\perp }u_{\alpha
}\cdot \nabla _{\perp }\phi\, dxdt+\gamma \int_{0}^{T}\int_{\Gamma _{\perp
}}u_{\alpha }\phi \,d\sigma \,dt=0,\quad \forall \phi \in \mathcal{D}.%
\end{array}%
\end{equation}%
By standard arguments, this problem is well posed. Indeed, consider the
mapping 
\begin{equation*}
\mathcal{T}:B_{R}(0)\rightarrow B_{R}(0),\quad B_{R}(0):=\{v\in
L^{2}(Q_{T})\,\,/\,\,||v||_{L^{2}(Q_{T})}\leq R\}\,,
\end{equation*}%
where we associate to any $v\in B_{R}(0)$ the unique solution $u\in
W_{2}^{1}(0,T;H^{1}(\Omega ),L^{2}(\Omega ))$ of the linearized, regular
parabolic problem 
\begin{eqnarray}
&&\int_{0}^{T}\langle \partial _{t}u(t,\cdot ),\phi (t,\cdot )\rangle
_{(H^{1})^{\ast },H^{1}}\,dt+\int_{0}^{T}\int_{\Omega }A_{||}a_{\alpha
}(v)\nabla _{||}u\cdot \nabla _{||}\phi\, dxdt  \notag \\
&& \hspace{1cm} +\int_{0}^{T}\int_{\Omega }A_{\perp }\nabla _{\perp }u\cdot \nabla _{\perp
}\phi\, dxdt+\gamma \int_{0}^{T}\int_{\Gamma _{\perp }}u\phi \,d\sigma
\,dt=0,\quad \forall \phi \in \mathcal{D}.  \notag
\end{eqnarray}%
Indeed, taking $R:=\sqrt{T}||u^{0}||_{2}$, the mapping $\mathcal{T}$ is
well-defined, continuous and $\mathcal{T}(B_{R}(0))$ is relatively compact
in $L^{2}(Q_{T})$. The continuity follows from the fact that for $%
v_{n}\rightarrow v$ in $L^{2}(Q_{T})$ and $v_{n}\rightharpoonup w$ in $%
W_{2}^{1}(0,T;H^{1}(\Omega ),L^{2}(\Omega ))$, the Lebesgue dominated
convergence theorem permits us to pass to the limit in the linearized term
of the variational formulation. By Schauder fixed point theorem, $\mathcal{T}
$ has a fixed point $\mathcal{T}(u)=u$, which provides a solution to (\ref%
{var}).

The solution $u_{\alpha }$ of problem (\ref{var}) satisfies $0\leq u_{\alpha}\leq
M$, provided we have $0\leq u^{0}\leq M$. Indeed, define $u_{\alpha }^{-}:=\min
(0,u_{\alpha })\leq 0$. Then one gets for the initial condition $u_{\alpha
}^{-}(0,\cdot )\equiv 0$. Taking $u_{\alpha }^{-}$ as the test function in
the variational formulation (\ref{var}), yields immediately 
\begin{equation*}
\begin{array}{lll}
\displaystyle{\frac{1}{2}}\int_{\Omega }|u_{\alpha }^{-}(T,x)|^{2}dx & + & %
\displaystyle\int_{0}^{T}\int_{\Omega }A_{||}a_{\alpha }(u_{\alpha
}^{-})|\nabla _{||}u_{\alpha }^{-}|^{2}dx\,dt \\[3mm] 
& + & \displaystyle\int_{0}^{T}\int_{\Omega }A_{\perp }|\nabla _{\perp
}u_{\alpha }^{-}|^{2}dxdt+\gamma \int_{0}^{T}\int_{\Gamma _{\perp
}}|u_{\alpha }^{-}|^{2}\,d\sigma \,d\tau =0,%
\end{array}%
\end{equation*}%
which shows that $u_{\alpha }^{-}(T,\cdot )\equiv 0$. Since the same
argument can be applied to any final time $T$, we have $u_{\alpha }\geq 0$
in $Q_{\infty }$.

To prove the estimate from above, define $u_{\alpha }^{+}:=\max (0,u_{\alpha
}-M)$. Observe that $u_{\alpha }^{+}(0,\cdot )\equiv 0$ and take $u_{\alpha }^{+}$
as the test function in the variational formulation (\ref{var}): 
\begin{equation*}
\begin{array}{lll}
\displaystyle{\frac{1}{2}}\int_{\Omega }|u_{\alpha }^{+}(T,x)|^{2}dx & + & %
\displaystyle\int_{0}^{T}\int_{\Omega }A_{||}a_{\alpha }(u_{\alpha })|\nabla
_{||}u_{\alpha }^{+}|^{2}\, dx\,dt \\[3mm] 
& + & \displaystyle\int_{0}^{T}\int_{\Omega }A_{\perp }|\nabla _{\perp
}u_{\alpha }^{+}|^{2}\,dxdt+\gamma \int_{0}^{T}\int_{\Gamma _{\perp
}}u_{\alpha }u_{\alpha }^{+}\,d\sigma \,d\tau =0,%
\end{array}%
\end{equation*}%
which shows that $u_{\alpha }^{+}(T,\cdot )\equiv 0$. Since again the same
argument can be applied to any final time, we have $u_{\alpha }\leq M$ in $%
Q_{\infty }.$

\underline{2nd step: A priori estimates}\\
In order to pass to the limit $\alpha \rightarrow 0$, we will need some a
priori estimates for the solution $u_{\alpha }$, independent of $\alpha $.
Taking $\phi =u_{\alpha }$ in the variational formulation (\ref{var}) yields 
\begin{equation*}
\begin{array}{l}
\displaystyle{\frac{1}{2}}\int_{\Omega }|u_{\alpha
}(T,x)|^{2}dx+\int_{0}^{T}\int_{\Omega }A_{||}a_{\alpha }(u_{\alpha
})|\nabla _{||}u_{\alpha }|^{2}\,dxdt \\[3mm] 
\displaystyle+\int_{0}^{T}\int_{\Omega }A_{\perp }|\nabla _{\perp }u_{\alpha
}|^{2}\,dxdt+\gamma \int_{0}^{T}\int_{\Gamma _{\perp }}|u_{\alpha
}|^{2}\,d\sigma \,dt={\frac{1}{2}}\int_{\Omega }|u^{0}(x)|^{2}dx,%
\end{array}%
\end{equation*}%
which implies 
\begin{equation}\label{alpest}
\begin{array}{ll}
||u_{\alpha }||_{L^{\infty }(0,T;L^{2}(\Omega ))}\leq
||u^{0}||_{L^{2}(\Omega )}, 
& \quad
\int_{0}^{T}\int_{\Omega }a_{\alpha
}(u_{\alpha })|\nabla _{||}u_{\alpha }|^{2}dx\,dt\leq
C||u^{0}||_{L^{2}(\Omega )}^{2}\,,
\\
||\nabla _{\perp }u_{\alpha }||_{L^{2}(Q_{T})}\leq C||u^{0}||_{L^{2}(\Omega
)}, 
& \quad
||u_{\alpha }||_{L^{2}([0,T]\times \Gamma _{\perp })}\leq
C||u^{0}||_{L^{2}(\Omega )},
\end{array}
\end{equation}
with some constant $C>0$.

Taking now $\phi =\Lambda _{\alpha }(u_{\alpha })$ in (\ref{var}), which is
permitted since $u_{\alpha }\in L^{\infty }(Q_{T}) \cap L^2(0,T;H^1(\Omega))$, yields 
\begin{equation*}
\begin{array}{l}
\displaystyle\int_{0}^{T}\langle \partial _{t}u_{\alpha },\Lambda _{\alpha
}(u_{\alpha })\rangle _{(H^{1})^{\ast },H^{1}}\,dt+\int_{0}^{T}\int_{\Omega
}A_{||}|\nabla _{||}\left( \Lambda _{\alpha }(u_{\alpha })\right) |^{2}dxdt \\%
[3mm] 
\displaystyle+\int_{0}^{T}\int_{\Omega }A_{\perp }a_{\alpha }(u_{\alpha
})|\nabla _{\perp }u_{\alpha }|^{2}dxdt+\gamma \int_{0}^{T}\int_{\Gamma
_{\perp }}u_{\alpha }\Lambda _{\alpha }(u_{\alpha })d\sigma \,d\tau =0\,.%
\end{array}%
\end{equation*}%
The first term can be rewritten as 
\begin{equation*}
\int_{0}^{T}\langle \partial _{t}u_{\alpha },\Lambda _{\alpha }(u_{\alpha
})\rangle _{(H^{1})^{\ast },H^{1}}\,dt=\int_{\Omega }\Psi _{\alpha
}(u_{\alpha }(T,x))\,dx-\int_{\Omega }\Psi _{\alpha }(u^{0}(x))dx\,,
\end{equation*}%
with $\Psi _{\alpha }(u):=\int_{0}^{u}\Lambda _{\alpha }(s)\,ds$. Due to the
facts that $0\leq u_{\alpha }\leq M$, $\Lambda _{\alpha }(u_{\alpha })\geq 0$
and $\Psi _{\alpha }(u_{\alpha })\geq 0$, we get 
\begin{eqnarray*}
\int_{0}^{T}\int_{\Omega }A_{||}|\nabla _{||}\left( \Lambda _{\alpha }(u_{\alpha
})\right) |^{2}dxdt &\leq & C\int_{\Omega }\Psi _{\alpha }(u^{0}(x))\,dx\leq 
\frac{C}{m(m+1)}\int_{\Omega }(\alpha +u^{0}(x))^{m+1}\,dx \\
&\leq &C\left( \alpha ^{m+1}+M^{m+1}\right) \,.
\end{eqnarray*}%
Thus, we have that the family $\{\nabla _{||}\Lambda _{\alpha }(u_{\alpha
})\}_{\alpha }$, $\alpha \in ]0,1]$ is bounded in $L^{2}(Q_{T})$. Moreover, $%
\{\nabla _{\perp }\Lambda _{\alpha }(u_{\alpha })\}_{\alpha }$ is also
bounded in $L^{2}(Q_{T})$, since $\nabla _{\perp }(\Lambda _{\alpha
}(u_{\alpha }))=a_{\alpha }(u_{\alpha })\nabla _{\perp }u_{\alpha }$ and $%
a_{\alpha }(u_{\alpha })$ is uniformly bounded by some positive constant. Hence, $\{\Lambda _{\alpha }(u_{\alpha
})\}_{\alpha }$ is bounded in $L^2(0,T;H^1(\Omega))$.

Let $V=H^1(\Omega)\cap L^\infty(\Omega)$ be the Banach space with the norm 
$||\cdot||_V = ||\cdot||_{H^{1}(\Omega )} + ||\cdot||_{L^{\infty}(\Omega )}$. 
For any $\phi $ in $L^{\infty}(0,T;V)$,
$$
\left\vert \int_{0}^{T}\langle \partial_{t}\Lambda _{\alpha }(u_{\alpha }),\phi
\rangle _{(H^{1})^{\ast },H^{1}}\,dt\right\vert 
=
\left\vert \int_{0}^{T}\langle \partial_{t} u_{\alpha },a_{\alpha }(u_{\alpha })\phi
\rangle _{(H^{1})^{\ast },H^{1}}\,dt\right\vert 
\le C ||\phi ||_{L^{\infty}(0,T;V)} 
$$
with a constant $C$ independent of $\alpha$. This follows from the variational formulation (\ref{var}) with 
$\phi$ replaced by $a_{\alpha }(u_{\alpha })\phi$ and from the estimates (\ref{alpest}).
We see thus that the family $\{\partial _{t}\Lambda _{\alpha }(u_{\alpha })\}_{\alpha }$ is bounded in 
$L^{1}(0,T;V^{\ast })$. We remind also that $\{\Lambda _{\alpha }(u_{\alpha})\}_{\alpha }$ is bounded in $L^2(0,T;H^1(\Omega))$.
Aubin-Simon compactness lemma \cite{Simon} applied to the triple of spaces 
$H^1(\Omega) \subset L^2(\Omega) \subset V^*$ allows us now to conclude that the set $\{\Lambda _{\alpha
}(u_{\alpha })\}_{\alpha }$ is relatively compact in $L^{2}(0,T;L^2(\Omega))=L^{2}(Q_T)$.

\underline{3rd step: Passage to the limit}\\
The aim now is to pass to the limit $\alpha \rightarrow 0$ in the
variational formulation (\ref{var}) in order to show the existence of a weak
solution of problem (\ref{P}). The a priori estimates of the last step
permit us to show, that there exists a function $u\in L^{2}(Q_{T})$ satisfying $%
0\leq u\leq M$ in $Q_{T}$ and such that after extracting a sub-sequence from $%
\{u_{\alpha }\}_{\alpha }$, we have
\begin{equation*}
u_{\alpha }\rightharpoonup _{\alpha \rightarrow 0}u\quad \text{in}\quad
L^{2}(Q_{T})\,,\quad u_{\alpha }|_{\Gamma _{\perp }}\rightharpoonup _{\alpha
\rightarrow 0}u|_{\Gamma _{\perp }}\quad \text{in}\quad L^{2}([0,T]\times
\Gamma _{\perp })\,,
\end{equation*}
\begin{equation*}
\nabla _{\perp }u_{\alpha }\rightharpoonup _{\alpha \rightarrow 0}\nabla
_{\perp }u\quad \text{in}\quad L^{2}(Q_{T})\,\quad \text{and\quad }\partial
_{t}u_{\alpha }\rightharpoonup _{\alpha \rightarrow 0}\partial _{t}u\quad 
\text{in}\quad L^{2}(0,T;(H^{1}(\Omega ))^{\ast })\,.
\end{equation*}%
To pass to the limit in the non-linear term, we use first the fact that $%
\{\Lambda _{\alpha }(u_{\alpha })\}_{\alpha }$ is bounded in $%
L^{2}(0,T;H^{1}(\Omega ))$, so that there exists some function $w\in
L^{2}(0,T;H^{1}(\Omega ))$ satisfying 
\begin{equation*}
\Lambda _{\alpha }(u_{\alpha })\rightharpoonup _{\alpha \rightarrow 0}w\quad 
\text{in}\quad L^{2}(0,T;H^{1}(\Omega ))\,.
\end{equation*}%
In order to identify the function $w$, we need some pointwise convergence of
the sequence $u_{\alpha }$. For this, recall that the sequence $\{\Lambda
_{\alpha }(u_{\alpha })\}_{\alpha }$ is relatively compact in $L^{2}(Q_{T})$.
Thus, up to a sub-sequence 
\begin{equation*}
\Lambda _{\alpha }(u_{\alpha })\rightarrow _{\alpha \rightarrow 0}w\quad 
\text{in}\quad L^{2}(Q_{T})\,,\quad \text{hence}\quad \quad \Lambda _{\alpha
}(u_{\alpha })\rightarrow _{\alpha \rightarrow 0}w\,,\quad \text{a.e. in}%
\quad Q_{T}\,.
\end{equation*}%
Since for any fixed $u\in [0,M]$, $\Lambda _{\alpha }(u)\rightarrow
_{\alpha \rightarrow 0} \Lambda(u) := \frac{1}{m}u^{m}$, we have $u_{\alpha }\rightarrow _{\alpha \rightarrow 0} \Lambda^{-1}w$ a.e. in
$Q_{T}$. This permits to identify the
function $\Lambda^{-1}w$ with $u$, so that $w=\frac{1}{m}u^{m}\in L^{2}(0,T;H^{1}(\Omega ))$.

All these convergences allow now to pass to the limit in the variational
formulation (\ref{var}) and to conclude the proof of existence for problem (%
\ref{Pm}).

\underline{4th step: Positivity and uniqueness}\\
Let $u$ be a weak solution to (\ref{Pm}).  
We use first the sub-solution constructed in Lemma \ref{subsol} and the comparison principle in Lemma \ref{compPr} to verify that $u\ge c e^{-Kt}$ with some positive constants $c$ small enough and $K$ large enough. We remark then that $M$ is a super-solution to (\ref{Pm}). Again, by Lemma \ref{compPr} we see that $u\le M$. 

Suppose now that (\ref{weak}) admits two solutions $u_{1}$ and $u_{2}$ in $%
\mathcal{W}$ with the same initial condition $u_{1}=u_{2}=u^{0}$ at $t=0$.
We know already that they are both strictly positive, so that Lemma \ref{compPr} implies $u_{1}\le
u_{2}$ on $Q_{\infty }$. Since the two solutions $u_{1}$ and $u_{2}$
are perfectly interchangeable in the above argument, we have also $u_{1}\geq
u_{2}$ and thus $u_{1}=u_{2}$ on $Q_{\infty }$.
\end{proof}
Let us remark here, that due to the strict positivity of the solution, in particular to the property that $u \ge c e^{-Kt}$ a.e. on $Q_{\infty}$, with some $c>0$ and $K>0$, we have
\begin{corollary}
Under the hypothesis of theorem \ref{mainTh}, the weak solution $u$ of (\ref{Pm}) belongs to the more regular space
\begin{align*}
\tilde{\mathcal{W}}:=\{& u\in L^{\infty }(Q_{\infty}),\text{ such that }\forall\, T>0 \,:\,\,\,
u\in L^{2}(0,T;H^{1}(\Omega ))\,,\,\,\, \partial
_{t}u\in L^{2}(0,T;(H^{1}(\Omega ))^{\ast })\}\,.
\end{align*}%
 Moreover, one has ${\frac{d}{dt}}||u(t,\cdot)||_{L^{2}(\Omega )}\leq 0$\,.
\end{corollary}

\section{Numerical method} \label{SEC3}

\subsection{Semi-discretization in space} \label{SEC32}

The singular perturbation problem (\ref{P}) is a highly anisotropic
equation. Its variational formulation reads: find $u(t,\cdot) \in \mathcal{V}:=H^1(\Omega)$ such that
\begin{eqnarray}
(P)\,\,\,
&&\langle \partial _{t}u(t,\cdot ),\,v \rangle
_{\mathcal{V}^{\ast },\mathcal{V}}+{1 \over \eps} \int_{\Omega }A_{||}|u|^{5/2}\nabla
_{||}u(t,\cdot )\cdot \nabla _{||}v\, dx  \label{weak_bis} \\
&& \hspace{2cm} +\int_{\Omega }A_{\perp }\nabla _{\perp }u(t,\cdot )\cdot \nabla _{\perp
}v\, dx+\gamma \int_{\Gamma _{\perp }}u(t,\cdot )v \,d\sigma
=0,\quad \forall v \in \mathcal{V}  \notag\,
\end{eqnarray}
for almost every $t \in (0,T)$. 
As mentioned already in Section \ref{SEC2}, this problem becomes ill-posed if we take formally the limit $\eps\to 0$. Indeed, only the leading term survives in this limit, so that any function from the space 
$$  
  {\cal G}:=\{p\in {\cal V}~/~\nabla _{\parallel}p=0 \text{ in }  \Omega \}
$$
would be a solution. It is easy to establish, however, the well-posed problem for the limit of the solutions to (P) as $\eps\to 0$. Indeed, one can restrain the test functions in (P) to be in the space ${\cal G}$ so that the $\eps$-dependent term disappears and the correct problem in the limit  $\eps\to 0$ reads: find $u(t,\cdot) \in \mathcal{G}$ such that
$$
(L)\,\,\,
\langle \partial _{t}u(t,\cdot ),\,v \rangle
_{\mathcal{V}^{\ast },\mathcal{V}}
 +\int_{\Omega }A_{\perp }\nabla _{\perp }u(t,\cdot )\cdot \nabla _{\perp
}v\, dx+\gamma \int_{\Gamma _{\perp }}u(t,\cdot )v \,d\sigma
=0,\quad \forall v \in \mathcal{G}  
$$
for almost every $t \in (0,T)$.

The discussion above shows that a straight-forward discretization of problem (P) may lead to very inaccurate results when $\eps << 1$. Indeed, setting $\eps=0$ would result in a singular problem, so that the problem with $\eps << 1$ would be very ill-conditioned. To cope with this difficulty and to obtain a numerical scheme which is uniformly accurate with respect to $\eps$, we introduce an Asymptotic-Preserving reformulation, very similar to the one introduced in \cite{DLNN}. The idea is to rewrite the singularly perturbed problem (\ref{weak_bis}) in an equivalent form, which is however well-posed when one sets there formally $\eps=0$ and gives moreover the correct limit problem (L). In order to do this, we introduce the auxiliary unknown $q$ by the relation $\eps\nabla_{||}q=u^{5/2}\nabla_{||}u$ in $\Omega$ and $q=0$ on $\Gamma_{in}$, which rescales the nasty part of the equation permitting to get rid of the terms of order $O(1/\eps)$. The reformulated problem, called in the sequel the Asymptotic-Preserving reformulation (AP-problem) reads: find $(u(t,\cdot),q(t,\cdot))\in {\cal V} \times {\cal L}$, solution of
\begin{equation}
  (AP)\,\,\, 
  \left\{ 
    \begin{array}{l}
      \ds \langle\frac{\partial u}{\partial t},\, v \rangle_{\mathcal{V}^{\ast },\mathcal{V}} + \int_{\Omega }(
      A_{\perp }\nabla _{\perp }u)\cdot \nabla _{\perp
      }v\,dx+\int_{\Omega } A_{\parallel} \nabla _{\parallel}q\cdot
      \nabla_{\parallel}v\,dx
      +\gamma \int_{\Gamma_{\perp}} uv\,ds
      =0, \\[1mm]
      \hspace{12cm}\forall v\in {\cal V}
      \\[1mm] \ds \int_{\Omega }A_{\parallel} u^{5/2}\nabla_{\parallel}u\cdot
      \nabla_{\parallel}w\,dx-\varepsilon \int_{\Omega } A_{\parallel} \nabla
      _{\parallel}q\cdot \nabla _{\parallel}w\,dx=0,\quad \forall w\in {\cal
        L}\,,
    \end{array}%
  \right.  \label{Pa}
\end{equation}
where 
\begin{gather}
  {\cal L}:=\{q\in L^{2}(\Omega )~/~\nabla _{\parallel}q\in
  L^{2}(\Omega )\text{ and }  q|_{\Gamma _{in}}=0\}. 
  \label{eq:Jobb}
\end{gather}
System (\ref{Pa}) is an equivalent reformulation (for fixed $\eps>0$)
of the original P-problem (\ref{weak_bis}). Putting now formally $\eps =0$ in (AP) leads to the well-posed limit problem
\begin{equation}
  (L')\,\,\, 
  \left\{ 
    \begin{array}{l}
      \ds \langle\frac{\partial u}{\partial t},\, v \rangle_{\mathcal{V}^{\ast },\mathcal{V}} + \int_{\Omega }(
      A_{\perp }\nabla _{\perp }u)\cdot \nabla _{\perp
      }v\,dx+\int_{\Omega } A_{\parallel} \nabla _{\parallel}q\cdot
      \nabla_{\parallel}v\,dx
      +\gamma \int_{\Gamma_{\perp}} uv\,ds
      =0,\\[1mm]
       \hspace{12cm} \forall v\in {\cal V}
      \\[1mm] \ds \int_{\Omega }A_{\parallel} u^{5/2}\nabla_{\parallel}u\cdot
      \nabla_{\parallel}w\,dx=0,\quad \forall w\in {\cal
        L}\,,
    \end{array}%
  \right.  \label{LP}
\end{equation}
which is equivalent to problem (L). Note that $q$ acts here as a Lagrange multiplier for the constraint $u\in\mathcal{G}$, which provides the uniqueness of the solution. Hence the AP-reformulation permits a continuous transition from the $P$-model to the $L$-model, which enables the uniform accuracy of the scheme with respect to $\eps$.

Let us now choose a triangulation of the domain $\Omega $ with triangles or quadrangles of order $h$ and introduce the finite element spaces ${\cal V}_{h} \subset {\mathcal V}$ and ${\cal L}_{h}\subset {\mathcal L}$ of
type $\mathbb P_{k}$ or $\mathbb Q_{k}$ on this mesh. The finite element 
discretization of (\ref{Pa}) writes then: find
$(u_{h},q_{h}) \in {\cal V}_{h} \times {\cal L}_h$ such that
\begin{equation}
  (AP)_h\,\,\, 
  \left\{ 
  \begin{array}{l}
    \ds \int_{\Omega } \frac{\partial u_h}{\partial t}v_h \, dx + \int_{\Omega }(
    A_{\perp }\nabla _{\perp }u_h)\cdot \nabla _{\perp
    }v_h\,dx+\int_{\Omega } A_{\parallel} \nabla _{\parallel}q_h\cdot
    \nabla_{\parallel}v_h\,dx
    +\gamma \int_{\Gamma_{\perp}} u_hv_h\,ds
    =0, \\[1mm]
    \hspace{12cm} \forall v_h\in {\cal V}_h
    \\[1mm] \ds \int_{\Omega }A_{\parallel} u_h^{5/2}\nabla_{\parallel}u_h\cdot
    \nabla_{\parallel}w_h\,dx-\varepsilon \int_{\Omega } A_{\parallel} \nabla
    _{\parallel}q_h\cdot \nabla _{\parallel}w_h\,dx=0, \quad \forall w\in {\cal
      L}_h\,.
  \end{array}%
  \right.
  \label{APh}
\end{equation}%
Remark that this system is continuous in time and also nonlinear, such that one has to develop now a procedure for the linearization and the discretization in time. This procedure has to be chosen carefully, such that the AP-property developed so far, is not destroyed. This is the aim of the next section.
\subsection{Semi-discretization in time} \label{SEC31}
In order to approach numerically the time derivative in (\ref{APh}), we introduce three different schemes : a standard first order,
implicit Euler scheme, the Crank-Nicolson scheme and a second
order, L-stable Runge-Kutta method. We show in the following that the
first order Euler-scheme is stable and asymptotic-preserving. The Crank-Nicolson schemes gives reliable
results and second order convergence under certain assumptions, but is
not asymptotic-preserving. Thus, if second order
accuracy in time is desired, the L-stable Runge-Kutta method has to be
applied. All three methods are exposed to numerical tests and
compared in Section \ref{sec:numres}.
\subsubsection{Implicit Euler scheme} 
\quad\\

Introducing the forms
\begin{gather}
  (\Theta,\chi):= \int_{\Omega} \Theta \chi \, dx 
  \,,   \label{eq:Jmbb0}\\
  a_{\parallel nl}(\Psi,\Theta,\chi):= \int_{\Omega} A_\parallel \Psi^{5/2} \nabla_\parallel \Theta \cdot
  \nabla_\parallel \chi \, dx 
  \,,   \label{eq:Jmbb1}\\
  a_\parallel(\Theta,\chi):= \int_{\Omega} A_\parallel \nabla_\parallel \Theta \cdot
  \nabla_\parallel \chi \, dx 
  \,, \quad  \quad a_{\perp}(\Theta,\chi):=
  \int_{\Omega} A_{\perp} \nabla_{\perp} \Theta \cdot \nabla_{\perp}
  \chi \, dx \,,
  \label{eq:Jmbb3}
\end{gather}
allows us to write the first order, implicit Euler method in the compact notation: Find $(u_h^{n+1},q_h^{n+1}) \in {\mathcal V}_h \times {\mathcal L}_h$, solution of
\begin{gather}
  (E_{AP}) \,\,\, 
  \left\{
  \begin{array}{l}
    (u_h^{n+1},v_h) + 
    \tau \left(a_\perp (u_h^{n+1},v_h) + a_\parallel (q_h^{n+1},v_h) 
    + \gamma \int_{\Gamma_{\perp}} u_h^{n+1}v_h \, ds\right)
    = (u_h^{n},v_h)
    \\[4mm]
    a_{\parallel nl} (u_h^{n},u_h^{n+1},w_h) - \eps a_\parallel
    (q_h^{n+1},w_h)
    = 0\,,
  \end{array}
  \right.
  \label{eq:Jpcb},
\end{gather}
where the non linear term $(u_h^{n+1})^{5/2}$ was replaced by a first order approximation in $\tau $ : 
\begin{gather}
  (u_h^{n+1})^{5/2} = (u_h^{n}+O(\tau ))^{5/2} = (u_h^{n})^{5/2} +O(\tau )
  \label{eq:Jqcb}.
\end{gather}
A slightly different first order AP-scheme was introduced in \cite{MN}
for the resolution of the same temperature balance problem. There, the
(P)-problem was firstly discretized in time (implicit Euler), then
linearized by a fixed point mapping, and finally the AP reformulation applied. The numerical results obtained in \cite{MN} are similar to the
present ones.

\subsubsection{Linearized Crank-Nicolson scheme}
\quad\\
To construct a scheme, which is second order in time, one can come to
the idea to employ the Crank-Nicolson scheme: Find
$(u_h^{n+1},q_h^{n+1}) \in {\mathcal V}_h \times {\mathcal L}_h$,
solution of
\begin{gather}
  \left\{
  \begin{array}{l}
    (u_h^{n+1},v_h) + 
    \tau \left(a_\perp (u_h^{n+1/2},v_h) + a_\parallel (q_h^{n+1},v_h) 
    + \gamma \int_{\Gamma_{\perp}} u_h^{n+1/2}v_h \, ds\right)
    =  (u_h^{n},v_h)
    \\[4mm]
    a_{\parallel nl} (u_h^{n+1/2},u_h^{n+1/2},w_h) - \eps a_\parallel
    (q_h^{n+1},w_h)
    = 0\,.
  \end{array}
  \right.
  \label{eq:Jtbb}
\end{gather}
As one can observe, we have to deal for each fixed $n$, with a
nonlinear equation. In the linear terms, one can set $u_h^{n+1/2} = \frac{1}{2}\left(u_h^{n+1}+u_h^{n}\right)$. To linearize the term $ a_{\parallel nl} (u_h^{n+1/2},u_h^{n+1/2},w_h)$ however, we shall use the standard
linear extrapolation method. In other words, the non-linearity in this last formula, $(u_h^{n+1/2})^{5/2}$, will be replaced by a linearized second
order approximation in $\tau $:
\begin{gather}
  (u_h^{n+1/2})^{5/2} = \left(u_h^{n}+ \frac{1}{2} \left(u_h^{n} -u_h^{n-1}
  \right) +O(\tau^2)\right)^{5/2} = \left(u_h^{n}+ \frac{1}{2} \left(u_h^{n} -u_h^{n-1}
  \right)\right)^{5/2}  +O(\tau ^2)
  \label{eq:Jvbb}.
\end{gather}
 The resulting linear system reads finally: Find $(u_h^{n+1},q_h^{n+1}) \in {\mathcal V}_h \times {\mathcal L}_h$, solution of
\begin{gather}
  (CN_{AP}) \,\,\, 
  \left\{
  \begin{array}{l}
    (u_h^{n+1},v_h) + 
    {\tau  \over 2} \left(a_\perp (u_h^{n+1},v_h) 
    + \gamma \int_{\Gamma_{\perp}} u_h^{n+1}v_h\, ds  \right) + \tau a_\parallel (q_h^{n+1},v_h)
    \\[3mm]
    \qquad \qquad =
    (u_h^{n},v_h)  
    -  {\tau  \over 2}\left(a_\perp (u_h^{n},v_h)  
    + \gamma \int_{\Gamma_{\perp}} u_h^{n}v_h \, ds \right) \,,
    \\[4mm]
    {1 \over 2} 
    a_{\parallel nl}
    \left(\frac{1}{2}\left(3u_h^{n}-u_h^{n-1}\right),u_h^{n+1},w_h\right)
    - \eps a_\parallel 
    (q_h^{n+1},w_h)\\[3mm]
    \qquad \qquad =
    -{1 \over 2}
    a_{\parallel nl}
    \left(\frac{1}{2}\left(3u_h^{n}-u_h^{n-1}\right),u_h^{n},w_h\right)\,.
  \end{array}
  \right.
\label{eq:Jqbb}
\end{gather}

Unfortunately, this method is not Asymptotic-Preserving. 
For
small values of $\eps$  one expects that
the solution will immediately fall into the space of functions almost
constant in the direction of the anisotropy, no matter what initial
condition was imposed. In the case of the Crank-Nicolson scheme for 
large time steps compared to $\eps/ (u^n_h)^{5/2}$, the
second equation in (\ref{eq:Jqbb}) will constrain the numerical solution to oscillate if the
initial condition is not already in the suitable space. This requires
the restrictive choice of a time step of the order of $\eps/ (u^n_h)^{5/2}$, which
yields the method inapplicable in general cases. In other words, the Crank-Nicolson scheme is unable to model diffusion processes for large $\Delta t$, due to the inadequate approximation of the damping processes. The Crank-Nicolson scheme is A-stable but not L-stable and the AP-property of a scheme is strongly related to the L-stability of the scheme.

As an example of the non-convergence of the ($CN_{AP}$) scheme in a general case,
we show some numerical results corresponding to a test case defined in the
Section \ref{sec:gauss}. The initial condition is a Gaussian peak
located in the center of the computational domain with a maximum of $10^5 K$. If the time step $\tau $ is too large, than $u^n_h$ will
immediately reach negative values and thus the numerical algorithm
will fail in the next iteration. However, if $\tau $ is sufficiently
small the ($CN_{AP}$) scheme is of second order in time. Unfortunately, the
biggest time step that does not provoke oscillations in the numerical
solution, is of the order of $10^{-16}s$, for an initial Gaussian peak
of $10^5 K$. This makes the ($CN_{AP}$) scheme of no practical use in real
simulations. These results are plotted on Figure \ref{fig:gauss_bis}.

\def\xxxa{0.47\textwidth}
\begin{figure}[!ht]
  \centering
  \subfigure[$\tau  = 0.1$]
  {\includegraphics[angle=0,width=\xxxa]{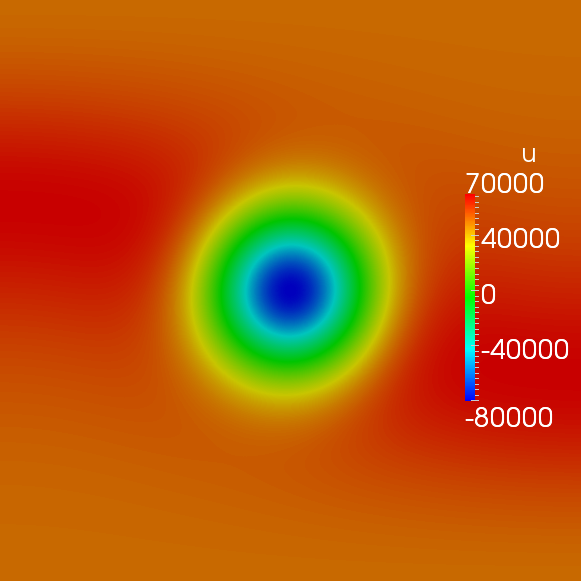}}
  \subfigure[$\tau  = 10^{-16}$]
  {\includegraphics[angle=0,width=\xxxa]{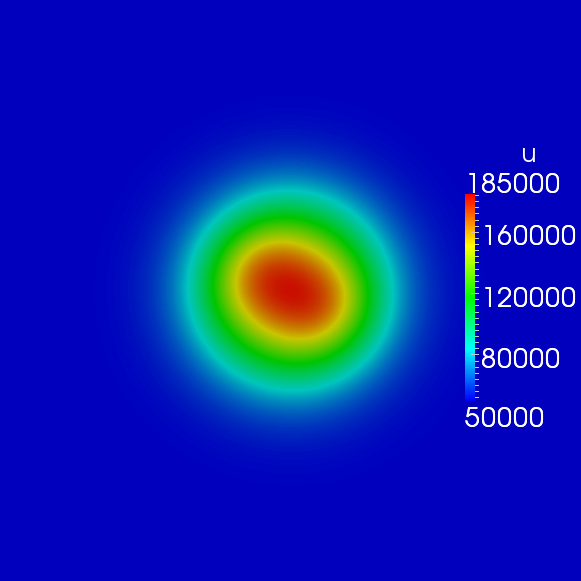}}

  \caption{Non convergence of the ($CN_{AP}$) scheme. Negative values of
    $u^n_h$ are obtained after one iteration of the method, for big
    time steps. If the time step is sufficiently small, the method
    converges.}
  \label{fig:gauss_bis}
\end{figure}

\subsubsection{L-stable Runge-Kutta method}
\quad\\

As we are interested in an AP-scheme, which is second order accurate in time, we propose  now a two stage Diagonally Implicit
Runge-Kutta (DIRK) second order scheme, which does not suffer from the
limitations of the Crank-Nicolson discretization. The scheme is
developed according to the following Butcher's diagram:
\begin{gather}
  \begin{array}{c|cc}
    \lambda & \lambda & 0 \\
    1 & 1-\lambda & \lambda \\
    \hline
    & 1-\lambda & \lambda 
  \end{array}
  \label{eq:Jrcb}
\end{gather}
with $\lambda = 1- {1 \over \sqrt{2}}$. 

\begin{remark}(Butcher's diagram)
  The coefficients of the $s$-stage Runge-Kutta method are usually displayed in
  a Butcher's diagram :
  \begin{gather}
    \begin{array}{c|ccc}
      c_1 & a_{11} & \cdots & a_{1s} \\
      \vdots & \vdots & & \vdots \\
      c_s & a_{s1} & \cdots & a_{ss}\\
      \hline
       & b_1 & \cdots & b_s
    \end{array}.
  \end{gather}
Applying this method to approximate to following problem
  \begin{gather}
    \frac{\partial u}{\partial t} = Lu + f(t)\,,
    \label{eq:J6cb}
  \end{gather}
 reads: For given $u^n$, being an
  approximation of $u(t_n)$, the $u^{n+1}$ is determined accordingly to :
  \begin{gather}
    u_i = u^{n} + \tau \sum_{j=1}^{s} a_{ij} (Lu_j + f(t+ c_j\tau)) ,
    \\
    u^{n+1} = u^n + \sum_{j=1}^{s} b_j u_j
    \label{eq:J8cb}.
  \end{gather}
  If $b_j = a_{sj}$ for $j=1,\ldots, s$ than $u^{n+1} = u_s$.
\end{remark}

The scheme (\ref{eq:Jrcb}) is known to be
L-stable, thus providing the Asymptotic Preserving property. The
scheme writes: Find $(u_h^{n+1},q_h^{n+1}) \in {\mathcal V}_h \times {\mathcal L}_h$, solution of
\begin{gather}
  (RK_{AP}) \,\,\, 
  \left.
  \begin{array}{l}
    \left\{
    \begin{array}{l}
      (u_{1,h}^{n+1},v_h)+ 
      \tau \lambda \left(a_\perp (u_{1,h}^{n+1},v_h) 
      + \gamma \int_{\Gamma_{\perp}} u_{1,h}^{n+1}v_h\, ds  + a_\parallel
      (q_{1,h}^{n+1},v_h) \right)      \\[3mm]
      \qquad \qquad \qquad = (u_{h}^{n},v_h)    
      \\[4mm]
      a_{\parallel nl}
      \left(u_{h}^{n}+\lambda (u_h^n - u_h^{n-1} ),u_{1,h}^{n+1},w_h\right)
      - \eps a_\parallel (q_{1,h}^{n+1},w_h)= 0
    \end{array}
    \right.
    \\
    \\
    \left\{
    \begin{array}{l}
      (u_{2,h}^{n+1},v_h)+ 
      \tau \lambda \left(a_\perp (u_{2,h}^{n+1},v_h) 
      + \gamma \int_{\Gamma_{\perp}} u_{2,h}^{n+1}v_h\, ds   + a_\parallel
      (q_{2,h}^{n+1},v_h) \right)      \\[3mm]
      \qquad \qquad \qquad 
      = (u_{h}^{n},v_h) + {1-\lambda \over \lambda } \left( u_{1,h}^{n+1}-u_{h}^{n},v_h\right)
      \\[4mm]
      a_{\parallel nl}
      \left(u_{h}^{n}+(u_h^n - u_h^{n-1} ),u_{2,h}^{n+1},w_h\right)
      - \eps a_\parallel (q_{2,h}^{n+1},w_h)= 0
    \end{array}
    \right.
    \\
    \\
    u_h^{n+1} = u_{2,h}^{n+1}\,, \qquad q_h^{n+1} = q_{2,h}^{n+1}\,,
    \end{array}
  \right.
  \label{eq:Jscb}
\end{gather}
with $u_{1,h}^{n+1}$ (respectively $u_{2,h}^{n+1}$) being the solution of the first
(respectively second) stage of the Runge-Kutta method. The terms $u_{h}^{n}+\lambda
(u_h^n - u_h^{n-1} )$ and $u_{h}^{n}+(u_h^n - u_h^{n-1} )$ are
respectively the second order time-approximations of $u_h(t+\lambda
\tau )$ and $u_h(t+\tau )$, used to linearize the problem.

For each time step we have therefore to assemble and solve two
linearized problems. This method is two times slower than the
Crank-Nicolson scheme, with the advantage however of maintaining the
AP-property of the scheme, advantage which is crucial for $0 < \eps
\ll 1$.
\section{Numerical results}\label{sec:numres}

In this section we compare the proposed implicit Euler-AP and DIRK-AP
schemes with a standard linearized implicit Euler discretization of
the initial singular perturbation problem (\ref{P}), given by
\begin{gather}
  (P)_{h\tau } \quad
  (u_h^{n+1},v_h) + 
  \tau \left(a_\perp (u_h^{n+1},v_h) + \frac{1}{\eps}a_{\parallel nl} (u_h^{n},u_h^{n+1},v_h) 
  + \gamma \int_{\Gamma_{\perp}} u_h^{n+1}v_h \, ds \right)
  = 
  (u_h^{n},v_h)  
  \label{eq:Jtcb}.
\end{gather}

\subsection{Discretization} \label{Discr}

Let us present the space discretization in a 2D case. We  
consider a square computational domain $\Omega =
[0,1]\times [0,1]$. All simulations are performed on structured
meshes. Let us introduce the Cartesian, homogeneous grid
\begin{gather}
  x_i = i / N_x \;\; , \;\; 0 \leq i \leq N_x \,, \quad
  y_j = j / N_y \;\; , \;\; 0 \leq j \leq N_y
  \label{eq:Jp8a},
\end{gather}
where $N_x$ and $N_y$ are positive even constants, corresponding to
the number of discretization intervals in the $x$-
resp. $y$-direction. The corresponding mesh-sizes are denoted by $h_x
>0$ resp. $h_y >0$. Choosing a $\mathbb Q_2$ finite element method
($\mathbb Q_2$-FEM), based on the following quadratic base functions

\begin{gather}
  \theta _{x_i}=
  \left\{
    \begin{array}{ll}
      \frac{(x-x_{i-2})(x-x_{i-1})}{2h_x^{2}} & x\in [x_{i-2},x_{i}],\\
      \frac{(x_{i+2}-x)(x_{i+1}-x)}{2h_x^{2}} & x\in [x_{i},x_{i+2}],\\
      0 & \text{else}
    \end{array}
  \right.\,, \quad 
  \theta _{y_j} =
  \left\{
    \begin{array}{ll}
      \frac{(y-y_{j-2})(y-y_{j-1})}{2h_y^{2}} & y\in [y_{j-2},y_{j}],\\
      \frac{(y_{j+2}-y)(y_{j+1}-y)}{2h_y^{2}} & y\in [y_{j},y_{j+2}],\\
      0 & \text{else}
    \end{array}
  \right.
  \label{eq:Js8a1}
\end{gather}
for even $i,j$  and
\begin{gather}
  \theta _{x_i}=
  \left\{
    \begin{array}{ll}
      \frac{(x_{i+1}-x)(x-x_{i-1})}{h_x^{2}} & x\in [x_{i-1},x_{i+1}],\\
      0 & \text{else}
    \end{array}
  \right.\,, \quad 
  \theta _{y_j} =
  \left\{
    \begin{array}{ll}
      \frac{(y_{j+1}-y)(y-y_{j-1})}{h_y^{2}} & y\in [y_{j-1},y_{j+1}],\\
      0 & \text{else}
    \end{array}
  \right.
  \label{eq:Js8a2}
\end{gather}
for odd $i,j$, we define the space
$$
W_h := \{ v_h = \sum_{i,j} v_{ij}\,  \theta_{x_i} (x)\,  \theta_{y_j}(y)\}\,.
$$
The spaces ${\cal V}_h$ and ${\cal L}_h$ are then defined by
\begin{equation*}
  {\cal V}_{h}={\cal W}_{h}\,, \quad {\cal L}_{h}=\{q_{h}\in {\cal
    V}_{h}\,, \text{ such that\,\,\,} q_{h}|_{\Gamma _{in}}=0\}. 
\end{equation*}
The matrix elements are computed using the 2D Gauss quadrature
formula, with 3 points in the $x$ and $y$ direction:
\begin{gather}
  \int_{-1}^{1}\int_{-1}^{1}f (\xi,\eta)\, d\xi\, d\eta =
  \sum_{i,j=-1}^{1} \omega _{i}\omega _{j} f (\xi_i,\eta_j)\,,
  \label{eq:Jk9a}
\end{gather}
where $\xi_0=\eta_0=0$, $\xi_{\pm 1}=\eta_{\pm 1}=\pm\sqrt {\frac{3}{5}}$,
$\omega _0 = 8/9$ and $\omega _{\pm 1} = 5/9$, which is exact for
polynomials of degree 5. Linear systems obtained for all methods in
these numerical experiments are solved using a LU decomposition,
implemented by the MUMPS library.

\subsection{Numerical tests}

\subsubsection{Known analytical solution}
\quad\\

First, let us construct a numerical test case with a known
solution. Finding an analytical solution for an arbitrary $b$-field
presents a considerable difficulty. In the previous paper
\cite{DDLNN}, we presented a way to find such a solution. Let us
recall briefly how to do this. The starting point is a limit solution
\begin{gather}
  u^{0} = \left( \cos \left(\pi y +\alpha (y^2-y)\cos (\pi x) \right) +
  4\right) T_m e^{-t}
  \label{eq:J79a},
\end{gather}
where $\alpha $ is a numerical constant aimed to control the
variations of $b$. For $\alpha =0$, the limit solution represents a
solution for the constant $b$ case. The parameter $T_m$ is the scaling
of $u_0$.

Since $u^{0}$ is a limit solution, it is constant along the $b$
field lines. Therefore we can determine the $b$ field using the
following implication
\begin{gather}
  \nabla_{\parallel} u^{0} = 0 \quad \Rightarrow \quad
  b_x \frac{\partial u^{0}}{\partial x} +
  b_y \frac{\partial u^{0}}{\partial y} = 0\,,
  \label{eq:J89a}
\end{gather}
which yields for example
\begin{gather}
  b = \frac{B}{|B|}\, , \quad
  B =
  \left(
  \begin{array}{c}
    \alpha  (2y-1) \cos (\pi x) + \pi \\
    \pi \alpha  (y^2-y) \sin (\pi x)
  \end{array}
  \right)
  \label{eq:J99a}\,\quad. 
\end{gather}
Note that the field $B$, constructed in this way, satisfies
$\text{div} B = 0$, which is an important property in the framework of
plasma simulations. Furthermore, we have $B \neq 0$ in the
computational domain. Now, we choose $u^\varepsilon $ to be a function
that converges, as $\varepsilon \rightarrow 0$, to the limit solution
$u^{0}$, for example
\begin{gather}
  p = \left( \cos \left(\pi y +\alpha (y^2-y)\cos (\pi x) \right) +
  4\right) T_m e^{-t}\\
  q = p^{-3/2} \sin (3\pi x) /3\pi \\
  u = p + \eps q
  \label{eq:Jrbb}.
\end{gather}
In our simulations we set $\alpha =1$ so that the direction of the
anisotropy is variable in the computational domain.
Note that we have
\begin{gather}
  -\frac{1}{\eps} u^{5/2}(t,0) \nabla_\parallel u(t,0) = u(t,0) \,, \\
   \frac{1}{\eps} u^{5/2}(t,1) \nabla_\parallel u(t,1) = u(t,1) 
  \label{eq:Jsbb}.
\end{gather}
The problem is supplied with a force term computed accordingly.

As an initial condition we take $u(t=0)$, with $u$ defined by
(\ref{eq:Jrbb}). In this setting we expect both Asymptotic-Preserving
methods ($E_{AP}$) and ($RK_{AP}$) to converge in the optimal rate,
independently on $\eps$ and $b$.

First we test the space convergence of the methods. To do this we
choose a small time step such that the time discretization error is
much smaller than the space discretization error. We then vary the
mesh size and perform simulations for 100 time steps. The results are
summarized in Table \ref{tab:conv_h1} and Figure
\ref{fig:conv_h1}. All three methods give as expected the third order
space convergence in the $L_2$-norm for large values of
$\eps$. Moreover, due to the extremely small time step, the numerical
precision is the same, even if one uses second or first order
methods. For small values of $\eps$ only the Asymptotic Preserving
schemes give good numerical solutions.

\def\xxxa{0.45\textwidth}
\begin{figure}[!ht]
  \centering
  \subfigure[$h = 0.1$]
  {\includegraphics[angle=0,width=\xxxa]{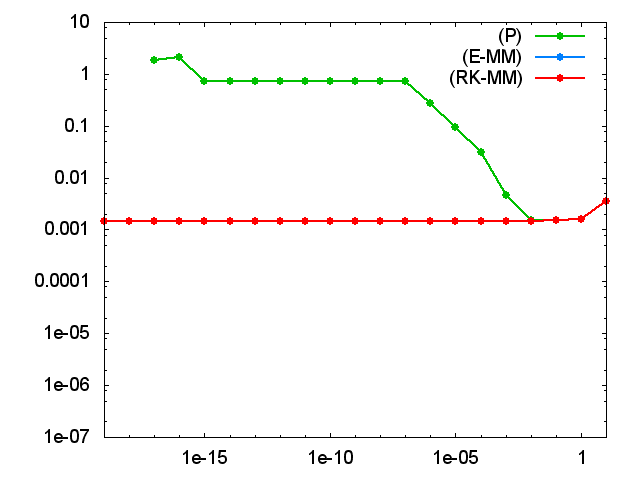}}
  \subfigure[$h = 0.00625$]
  {\includegraphics[angle=0,width=\xxxa]{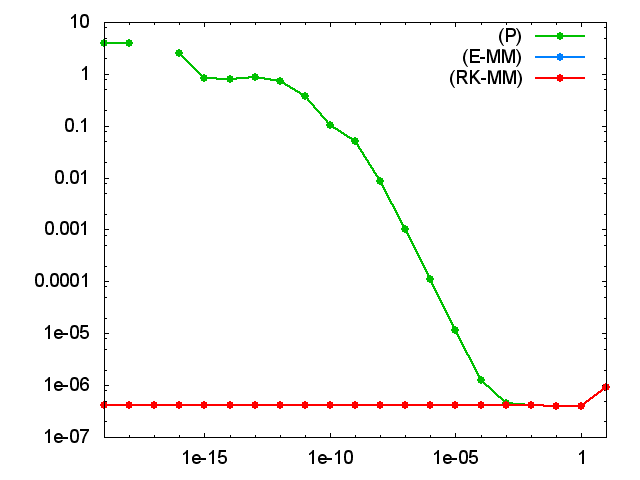}}

  \caption{Relative $L^{2}$-errors between the exact
    solution $u^{\varepsilon }$ and the computed solution for the standard
    scheme (P), Euler AP method ($E_{AP}$) and DIRK AP scheme ($RK_{AP}$)
    as a function of $\eps$
    and for $h = 0.1$ resp.  $h=0.00625$. The time step is $\tau=10^{-6}$.}
  \label{fig:conv_h1}
\end{figure}

\begin{table}
  \centering
  \begin{tabular}{|c||c|c|c|}
    \hline
    \multirow{2}{*}{$h$}  
    & \multicolumn{3}{|c|}{\rule{0pt}{2.5ex}$L^2$-error\qquad $\eps=1$} \\
    \cline{2-4} 
    &\rule{0pt}{2.5ex}
    P& $E_{AP}$& $RK_{AP}$\\
    \hline
    \hline\rule{0pt}{2.5ex}
    0.1 &
    $1.60\times 10^{-3}$ &
    $1.60\times 10^{-3}$ &
    $1.60\times 10^{-3}$ 
    \\
    \hline\rule{0pt}{2.5ex}
    0.05 &
    $2.02\times 10^{-4}$ &
    $2.02\times 10^{-4}$ &
    $2.02\times 10^{-4}$ 
    \\
    \hline\rule{0pt}{2.5ex}
    0.025 &
    $2.55\times 10^{-5}$ &
    $2.55\times 10^{-5}$ &
    $2.55\times 10^{-5}$ 
    \\
    \hline\rule{0pt}{2.5ex}
    0.0125 &
    $3.2\times 10^{-6}$ &
    $3.2\times 10^{-6}$ &
    $3.2\times 10^{-6}$ 
    \\
    \hline\rule{0pt}{2.5ex}
    0.00625 &
    $4.0\times 10^{-7}$ &
    $4.0\times 10^{-7}$ &
    $4.0\times 10^{-7}$ 
    \\
    \hline
  \end{tabular}
  \begin{tabular}{|c||c|c|c|}
    \hline
    \multirow{2}{*}{$h$}  
    & \multicolumn{3}{|c|}{\rule{0pt}{2.5ex}$L^2$-error\qquad $\eps=10^{-10}$} \\
    \cline{2-4} 
    &\rule{0pt}{2.5ex}
    P& $E_{AP}$& $RK_{AP}$\\
    \hline
    \hline\rule{0pt}{2.5ex}
    0.1 &
    $7.3\times 10^{-1}$ &
    $1.47\times 10^{-3}$ &
    $1.47\times 10^{-3}$ 
    \\
    \hline\rule{0pt}{2.5ex}
    0.05 &
    $7.3\times 10^{-1}$ &
    $2.04\times 10^{-4}$ &
    $2.04\times 10^{-4}$ 
    \\
    \hline\rule{0pt}{2.5ex}
    0.025 &
    $7.3\times 10^{-1}$ &
    $2.65\times 10^{-5}$ &
    $2.65\times 10^{-5}$ 
    \\
    \hline\rule{0pt}{2.5ex}
    0.0125 &
    $4.9\times 10^{-1}$ &
    $3.3\times 10^{-6}$ &
    $3.3\times 10^{-6}$ 
    \\
    \hline\rule{0pt}{2.5ex}
    0.00625 &
    $1.04\times 10^{-1}$ &
    $4.2\times 10^{-7}$ &
    $4.2\times 10^{-7}$ 
    \\
    \hline
  \end{tabular}
 
  \caption{The absolute error of $u$ in the $L^{2}$-norm for different mesh
    sizes and $\eps =1$ or $\eps =10^{-10}$, using the singular perturbation scheme (P) and the
    two proposed AP-schemes for a time step of $\tau = 10^{-6}s$ and at instant
    $t=10^{-4}$, with $T_m = 1$.  }
  \label{tab:conv_h1}
\end{table}

Finally we test the time convergence of the methods. To do this we
choose a small mesh size such that the space discretization error is
smaller than the time discretization error. We then vary the time step
and perform simulations on a fixed grid. The results are summarized in
Table \ref{tab:conv_t} and Figures \ref{fig:conv_t1} and
\ref{fig:conv_t2}. Note that the $(RK_{AP})$ scheme is of second order
in time as long as the error due to the time discretization dominates
the error induced by the space discretization. The standard (P)-scheme works well and is of first order, as long as $\eps$ is close
to one. The $(E_{AP})$ scheme is of first order for all values of the
anisotropic parameter. Also note that while the $(RK_{AP})$ scheme demands
twice more computational time than the $(E_{AP})$ scheme, it gives much
better precision. In order to achieve a relative error of the order of
$10^{-4}$ for $\eps=1$ it suffices to take a time step of $\tau =0.05$ in
the RK-scheme. A comparable accuracy with $(E_{AP})$ is obtained for a time
step 16 times smaller. In the case of $\eps = 10^{-10}$ the ratio is
32.

\def\xxxa{0.45\textwidth}
\begin{figure}[!ht]
  \centering
  {\includegraphics[angle=0,width=\xxxa]{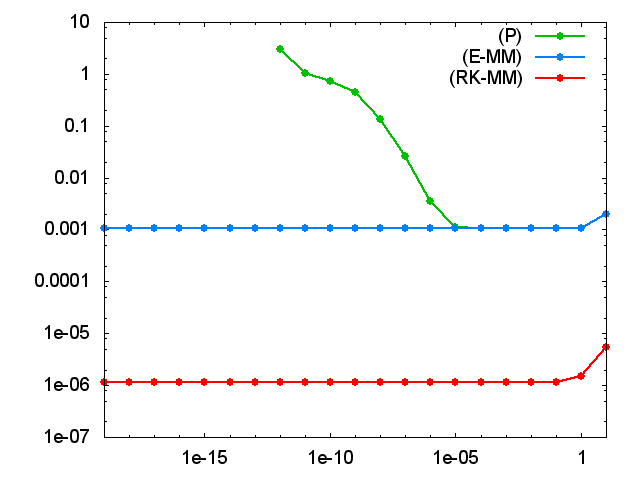}}

  \caption{Relative $L^{2}$-errors between the exact
    solution $u^{\varepsilon }$ and the computed solution with the standard
    scheme (P), the Euler-AP method ($E_{AP}$) and the DIRK-AP scheme ($RK_{AP}$)
    as a function of $\eps$ and for $\tau  = 0.00625$. The spacial grid is $200 \times 200$.}
  \label{fig:conv_t1}
\end{figure}

\def\xxxa{0.45\textwidth}
\begin{figure}[!ht]
  \centering
  \subfigure[$\eps=1$]
  {\includegraphics[angle=0,width=\xxxa]{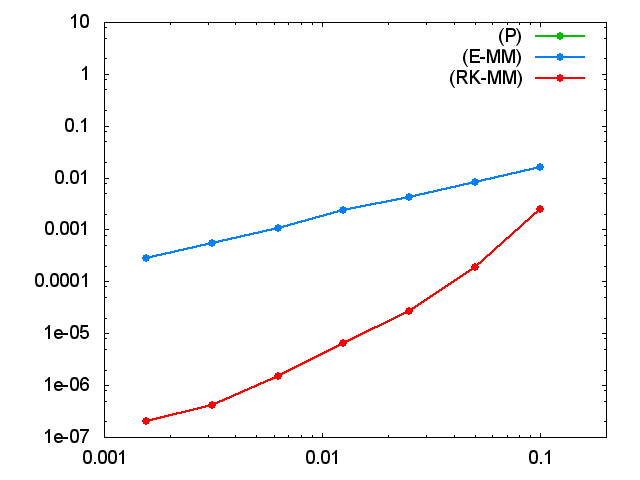}}
  \subfigure[$\eps=10^{-10}$]
  {\includegraphics[angle=0,width=\xxxa]{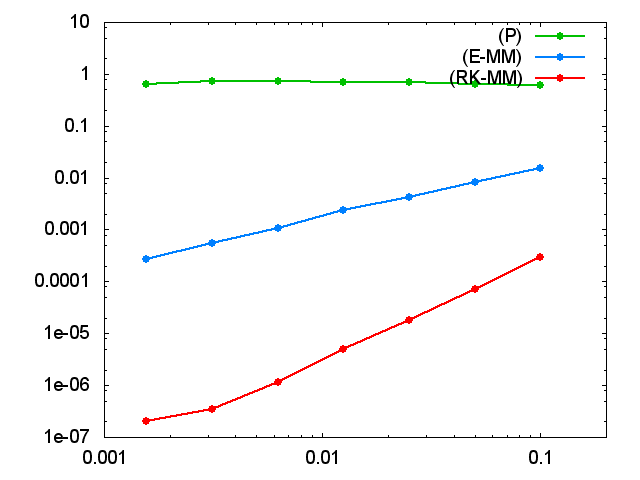}}

  \caption{Relative $L^{2}$-errors between the exact solution
    $u^{\varepsilon }$ and the computed solution with the standard scheme
    (P), the Euler-AP method ($E_{AP}$) and the DIRK-AP scheme ($RK_{AP}$) as a
    function of $\tau$, and for $\eps=1$ resp. $\eps=10^{-10}$ and a mesh with
    $200\times 200$ points. Note that for $\eps =1$ the P-scheme and the $E_{AP}$
    scheme give the same precision.}
  \label{fig:conv_t2}
\end{figure}

\begin{table}
  \centering
  \begin{tabular}{|c||c|c|c|}
    \hline
    \multirow{2}{*}{$\tau $}  
    & \multicolumn{3}{|c|}{\rule{0pt}{2.5ex}$L^2$-error\qquad $\eps=1$} \\
    \cline{2-4} 
    &\rule{0pt}{2.5ex}
    P& $E_{AP}$& $RK_{AP}$\\
    \hline
    \hline\rule{0pt}{2.5ex}
    0.1 &$1.57\times 10^{-2} $&$ 1.57\times 10^{-2} $&$ 2.52\times 10^{-3} $
    \\
    \hline\rule{0pt}{2.5ex}
    0.05 &$8.28\times 10^{-3} $&$ 8.28\times 10^{-3} $&$ 1.93\times 10^{-4} $ 
    \\
    \hline\rule{0pt}{2.5ex}
    0.025 &$4.25\times 10^{-3} $&$ 4.25\times 10^{-3} $&$ 2.62\times 10^{-5} $
    \\
    \hline\rule{0pt}{2.5ex}
    0.0125 &$2.37\times 10^{-3} $&$ 2.37\times 10^{-3} $&$ 6.54\times 10^{-6} $
    \\
    \hline\rule{0pt}{2.5ex}
    0.00625 &$1.08\times 10^{-3} $&$ 1.08\times 10^{-3} $&$ 1.50\times 10^{-6} $
    \\
    \hline\rule{0pt}{2.5ex}
    0.003125 &$5.44\times 10^{-4} $&$ 5.44\times 10^{-4} $&$ 4.08\times 10^{-7} $
    \\
    \hline\rule{0pt}{2.5ex}
    0.0015625 &$2.76\times 10^{-4} $&$ 2.76\times 10^{-4} $&$ 2.07\times 10^{-7} $
    \\
    \hline
  \end{tabular}
  \begin{tabular}{|c||c|c|c|}
    \hline
    \multirow{2}{*}{$\tau $}  
    & \multicolumn{3}{|c|}{\rule{0pt}{2.5ex}$L^2$-error\qquad $\eps=10^{-10}$} \\
    \cline{2-4} 
    &\rule{0pt}{2.5ex}
    P& $E_{AP}$& $RK_{AP}$\\
    \hline
    \hline\rule{0pt}{2.5ex}
    0.1 &$6.14\times 10^{-1} $&$ 1.57\times 10^{-2} $&$ 2.90\times 10^{-4} $
    \\
    \hline\rule{0pt}{2.5ex}
    0.05 &$6.30\times 10^{-1} $&$ 8.22\times 10^{-3} $&$ 7.21\times 10^{-5} $
    \\
    \hline\rule{0pt}{2.5ex}
    0.025 &$6.92\times 10^{-1} $&$ 4.22\times 10^{-3} $&$ 1.80\times 10^{-5} $
    \\
    \hline\rule{0pt}{2.5ex}
    0.0125 &$7.08\times 10^{-1} $&$ 2.36\times 10^{-3} $&$ 4.91\times 10^{-6} $
    \\
    \hline\rule{0pt}{2.5ex}
    0.00625 &$7.26\times 10^{-1} $&$ 1.08\times 10^{-3} $&$ 1.15\times 10^{-6} $
    \\
    \hline\rule{0pt}{2.5ex}
    0.003125 &$7.42\times 10^{-1} $&$ 5.40\times 10^{-4} $&$ 3.43\times 10^{-7} $
    \\
    \hline\rule{0pt}{2.5ex}
    0.0015625 &$6.42\times 10^{-1} $&$ 2.74\times 10^{-4} $&$ 2.05\times 10^{-7} $
    \\
    \hline
  \end{tabular}
  \caption{The absolute error of $u$ in the $L^{2}$-norm 
    for different time step using the singular
    perturbation scheme (P) and two proposed AP-schemes for mesh size
    $200\times 200$ at time $t=0.1s$ with $T_m = 1$.
  }
  \label{tab:conv_t}
\end{table}

To conclude, one can remark that the asymptotic-preserving schemes,
($E_{AP}$) and ($RK_{AP}$), are uniformly accurate with respect to the
perturbation parameter $\eps$. This essential feature can be very
useful in situations where the anisotropy  is variable in space, {\it
  i.e.} the parameter $\eps(x)$ is $x$-dependent. No mesh-adaptation
is any more needed in these cases, a simple Cartesian grid enables
accurate results, with no regard to the $\eps$-values.

\subsubsection{Initial Gaussian peak}\label{sec:gauss}
\quad\\

The second investigated test is the evolution of the following initial
Gaussian peak, located in the middle of the computational domain:
\begin{gather}
  u(t=0) =
  {T_m\over 2}
  \left(
  1 + e^{-50 (x-0.5)^2-50(y-0.5)^2}
  \right)\,,
  \label{eq:Jucb}
\end{gather}
where $T_m=10^5\,K$ is the maximal temperature in the domain and the
anisotropy direction is given as in the previous tests. We perform
numerical experiments with the choice of $\eps =1$. We choose the time
step $\tau =0.01$ and perform numerical simulations on a fixed
$50\times 50$ grid with the final time set to $15s$. The time step is
big compared to the time scale induced by the initial
condition. Indeed, after the first iteration of the algorithm the
numerical solution immediately falls into the space of functions which
are almost constant in the direction of the anisotropy (see Figure
\ref{fig:gauss1}).  The evolution of the numerical solution consists
of two phases. The first one, in which the parallel components of the
diffusion operator dominate, is characterized by the exponential decay
of $||u_h||_{L_2(\Omega )}$, $\min(u_h)$ and $\max(u_h)$ (see Figure
\ref{fig:gauss}). When $u_h$ reaches some critical value, the parallel
part of the diffusion operator becomes smaller than the perpendicular
one. The direction of the strong diffusion is now inverted and the
numerical solution aligns itself rather with the perpendicular
direction. The minimum, maximum as well as the $L^2$-norm of $u_h$
continue to approach zero, but the decay is no longer exponential. The
$L^2$-norm and the maximal value remain close to each other and almost
constant in time. The minimal value of $u_h$, as well as the
boundary-values decrease much faster.

\def\xxxa{0.45\textwidth}
\begin{figure}[!ht]
  \centering
  {\includegraphics[angle=0,width=\xxxa]{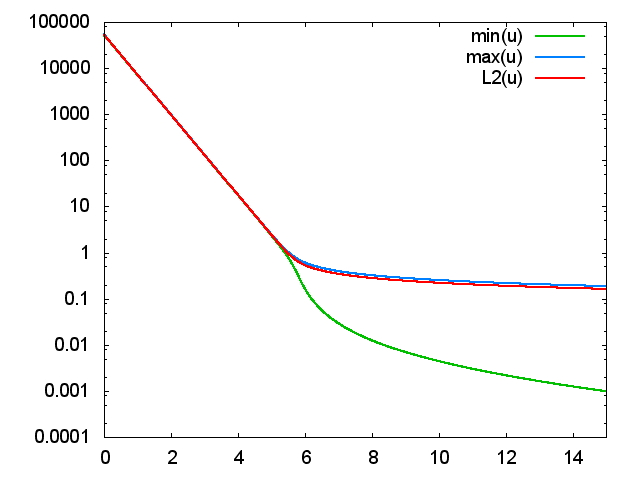}}

  \caption{$\min(u_h)$, $\max(u_h)$ and $||u_h||_{L^2(\Omega )}$ as a
    function of time for the Gaussian peak experiment, for $T_m = 10^5$
    and $\eps=1$. Time step is $\tau=0.01s$ and a mesh size of $50\times 50$. }
  \label{fig:gauss}
\end{figure}

\section{Conclusion}\label{SEC6}

The here presented Asymptotic-Preserving scheme proves to be an
efficient, general and easy to implement numerical method for solving
nonlinear, strongly anisotropic parabolic problems. This kind of
problems occur in several important applications, as for example
magnetically confined fusion plasmas. The method is based on a
reformulation of the problem, initially introduced by the
authors in an elliptic framework, and a careful linearization as well
as time-discretization of the resulting equation, which does not
destroy the AP-properties of the space-discretization. Numerical
experiments show clearly the advantages of such an AP-scheme.

\section*{Acknowledgments}
The authors would like to thank Michel Pierre and Giacomo Dimarco for
useful discussions. This work has been partially supported by the ANR
project ESPOIR (Edge Simulation of the Physics Of Iter Relevant
turbulent transport, 2009-2013) and the ANR project BOOST (Building
the future Of numerical methOdS for iTer, 2010-2014).
 
\newpage
\vfill
\def\xxxa{0.35\textwidth}
\begin{figure}[H]
  \centering
  \subfigure[$t = 0$]
  {\includegraphics[angle=0,width=\xxxa]{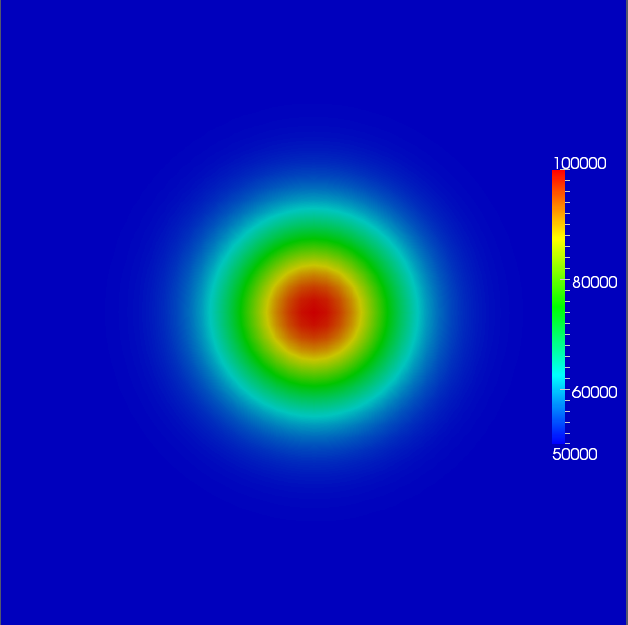}}
  \subfigure[$t = 0.01$]
  {\includegraphics[angle=0,width=\xxxa]{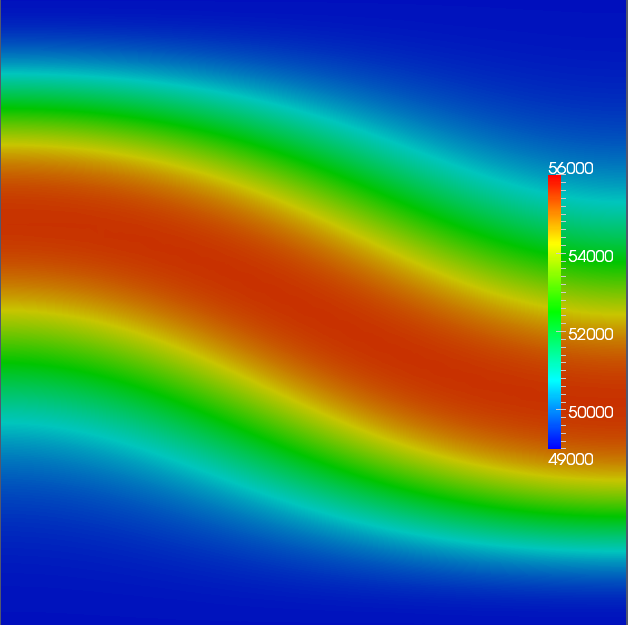}}
  \subfigure[$t = 4.5$]
  {\includegraphics[angle=0,width=\xxxa]{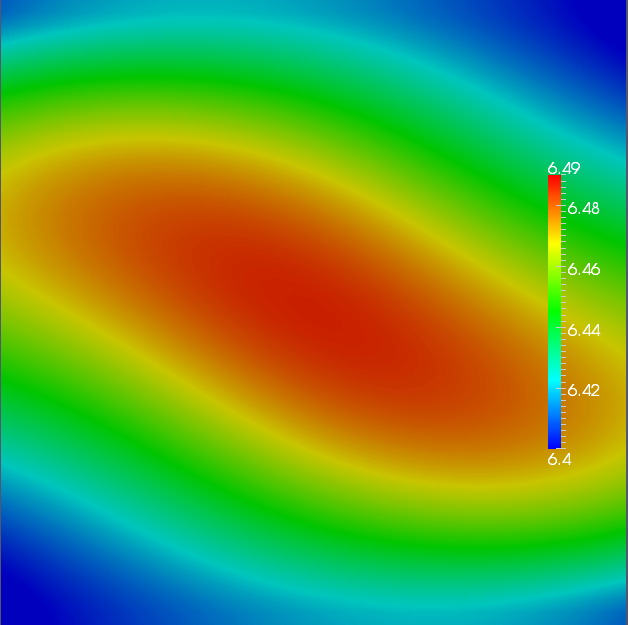}}
  \subfigure[$t = 4.75$]
  {\includegraphics[angle=0,width=\xxxa]{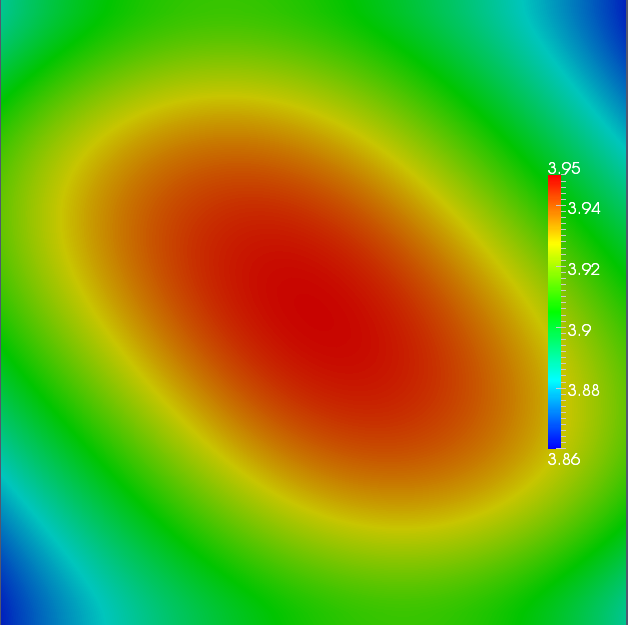}}
  \subfigure[$t = 5$]
  {\includegraphics[angle=0,width=\xxxa]{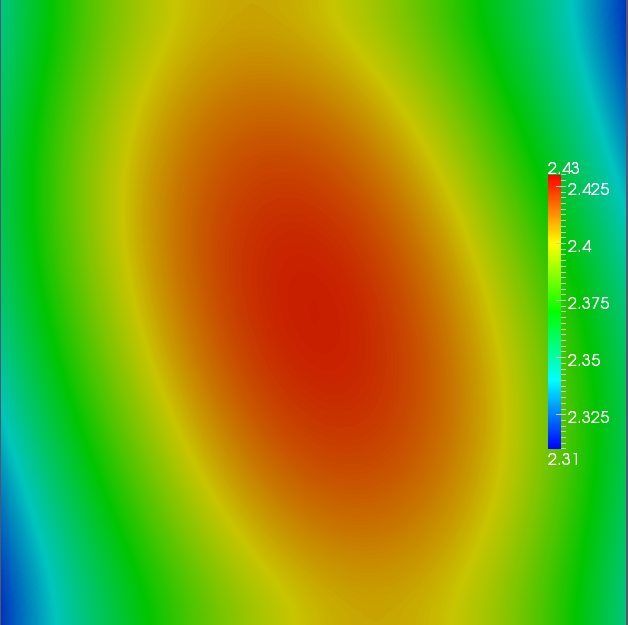}}
  \subfigure[$t = 6$]
  {\includegraphics[angle=0,width=\xxxa]{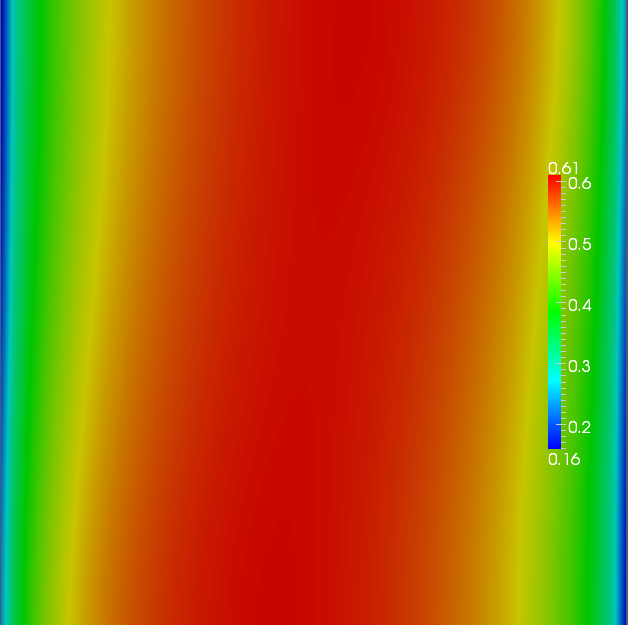}}

  \caption{Numerical solution at different time steps
    for the Gaussian peak experiment, for $T_m = 10^5$
    and $\eps=1$. Time step is $\tau=0.01s$ and a mesh size of $50\times 50$. }
  \label{fig:gauss1}
\end{figure}
\newpage
\bibliographystyle{abbrv}
\bibliography{bib_aniso}

\end{document}